\documentclass[10pt]{amsart}
\usepackage[margin=3cm]{geometry}
\usepackage{amssymb,latexsym,dsfont}
\usepackage[utf8]{inputenc}
\usepackage{graphicx}
\usepackage{xcolor}

\def\a        {{\boldsymbol a}}
\def\x        {{\boldsymbol x}}
\def\n        {{\boldsymbol n}}
\def\dx     {{\rm d}{\boldsymbol x}}

\def\R {{\mathds R}}
\def\N {{\mathds N}}

\newcommand{\rf}[1]{{\color{blue} #1}}

\newtheorem{theorem}{Theorem}[section]
\newtheorem{lemma}[theorem]{Lemma}
\newtheorem{proposition}[theorem]{Proposition}
\newtheorem{corollary}[theorem]{Corollary}

\newtheorem{remark}[theorem]{Remark}
\begin{document}

\title[Numerical analysis for the classical Keller--Segel model]{Analysis of a fully discrete approximation for the classical Keller--Segel model: lower and \emph{a priori} bounds}
\author[J. V. Gutiérrez-Santacreu]{Juan Vicente Gutiérrez-Santacreu$^\dag$}
\address{$\dag$Dpto. de Matemática Aplicada I\\
         E. T. S. I. Informática\\
         Universidad de Sevilla\\
         Avda. Reina Mercedes, s/n.\\
         E-41012 Sevilla\\
         Spain\\
         E-mail: {\tt juanvi@us.es}}
\author[J. R. Rodríguez-Galván]{José Rafael Rodríguez-Galván$^\ddag$}
\address{$\ddag$ Departamento de Matemáticas. Facultad de Ciencias. Campus Universitario de Puerto Real, Universidad de Cádiz. E-11510 Puerto Real, Cádiz E-mail: {\tt rafael.rodriguez@uca.es}}
\thanks{JVGS and JRRG were partially supported by the Spanish Grant No. PGC2018-098308-B-I00 from Ministerio de Ciencias e Innovación - Agencia Estatal de Investigación with the participation of FEDER}

\date{\today}

\begin{abstract}
This paper is devoted to constructing approximate solutions for the classical Keller--Segel model governing \emph{chemotaxis}. It consists of a system of nonlinear parabolic equations, where the unknowns are the average density of cells (or organisms), which is a conserved variable,  and  the average density of chemoattractant.

The numerical proposal is made up of a crude finite element method together with a mass lumping technique and a semi-implicit Euler time integration. The resulting scheme turns out to be linear and decouples the computation of variables. The approximate solutions keep lower bounds -- positivity for the cell density and nonnegativity for the chemoattractant density --, are bounded in the $L^1(\Omega)$-norm, satisfy a discrete energy law, and have \emph{ a priori} energy estimates.  The latter is achieved by means of a discrete Moser--Trudinger inequality. As far as we know, our numerical method is the first one that can be encountered in the literature dealing with all of the previously mentioned properties at the same time. Furthermore, some numerical examples are carried out to support and complement the theoretical results.

\end{abstract}

\maketitle


\tableofcontents

\section{Introduction}
\subsection{Aims} In 1970/71 Keller and Segel \cite{Keller_Segel_1970, Keller_Segel_1971} attempted to derive a set of equations for modeling \emph{chemotaxis} -- a  biological process through which an organism (or a cell) migrates in response to a chemical stimulus being attractant or repellent.  It is nowadays well-known that the work of Keller and Segel turned out to be somehow biologically inaccurate since their equations provide unrealistic solutions; a little more precisely, solutions that blow up in finite time. Such a phenomenon does not occur in nature. Even though the original Keller--Segel equations are less relevant from a biological point of view, they are mathematically of great interest.

Much of work for the Keller--Segel equations has been carried out in developing purely analytical results, whereas there are very few numerical results in the literature. This is due to the fact that solving numerically the Keller--Segel equations is a challenging task because their solutions exhibit many interesting mathematical properties which are not easily adapted to a discrete framework. For instance, solutions to the Keller--Segel equations satisfy lower bounds (positivity and non-negativity) and enjoy an energy law, which is obtained by testing the equations against non linear functions. Cross-diffusion mechanisms governing the chemotactic phenomena are the responsible for the fact that the Keller--Segel equations are  so difficult to analyze not only theoretically but also numerically.

In spite of being a limited model, it is hoped that developing and analyzing numerical methods for the classical Keller--Segel equations may open new roads to deeper insights and better understandings for dealing with the numerical approximation of other chemotaxis models -- biologically more realistic, but which are, however, inspired on the original Keller--Segel formulation. In a nutshell, these other chemotaxis models are modifications of the Keller--Segel equations in order to avoid the non-physical blow up of solutions and hence produce solutions being closer to chemotaxis phenomena. For these other chemotaxis models, it is recommended the excellent surveys of Hillen and Painter \cite{Hillen_Painter_2009}, Horstamann \cite{Horstamann_2003, Horstamann_2004}, and, more recently, Bellomo, Bellouquid, Tao, and Winkler \cite{Bellomo_Bellouquid_Tao_Winkler_2015}. In these surveys the authors reviewed to date as to when they were written the state of art of modeling and mathematical analysis for the Keller--Segel equations and their variants.

It is our aim in this work to design a fully discrete algorithm for the classical Keller--Segel equations based on a finite element discretization whose discrete solutions satisfy lower and \emph{a priori} bounds. 

\subsection{The Keller--Segel equations} Let $\Omega\subset \R^2$ be a bounded domain, with $\n$ being its outward-directed unit normal vector to $\Omega$, and let $[0,T]$ be a time interval. Take $Q=(0,T]\times \Omega$ and $\Sigma=(0,T]\times\partial\Omega$. Then the boundary-value problem for the Keller--Segel equations reads a follows. Find $u: \bar Q\to (0,\infty)$ and $v:\bar Q \to [0,\infty)$ satisfying
\begin{equation}\label{KS}
\left\{
\begin{array}{rcll}
\partial_tu-\Delta u&=&-\nabla\cdot(u\nabla v)&\mbox{ in } Q,
\\
\partial_t v -\Delta v&=&u-v&\mbox{ in }Q,
\end{array}
\right.
\end{equation}
subject to the initial conditions
\begin{equation}\label{IC}
u(0)=u_0\quad\mbox{ and }\quad v(0)=v_0\quad\mbox{ in }\quad \Omega,
\end{equation}
and the boundary conditions
\begin{equation}\label{BC}
\nabla u\cdot \n=0\quad\mbox{ and }\quad \nabla v\cdot\n=0\quad\mbox{ on }\quad \Sigma.
\end{equation}
Here $u$ is the average density of organisms (or cells), which is a conserved variable, and $v$ is the average density of chemical sign, which is a nonconserved variable.

System \eqref{KS} was motivated by Keller and Segel \cite{Keller_Segel_1970} describing the aggregation phenomena exhibited by the amoeba \emph{Dictyostelium discoideum} due to an attractive chemical substance referred to as \emph{chemoattractant}, which is generated by the own amoeba and is, nevertheless, degraded by living conditions. Moreover, diffusion is also presented in the motion of amebae and chemoattractant.

The diffusion phenomena  performed by cells and chemoattractant are modelled by the terms  $-\Delta u$ and $-\Delta v$, respectively, whereas the aggregation mechanism is described by the term $-\nabla \cdot (u\nabla v)$. It is this nonlinear term that is the major difficulty in studying system \eqref{KS}.  Further the production and degradation of chemoattractant are associated with the term $u-v$.

Concerning the mathematical analysis for system \eqref{KS}, Nagai, Senba, and Yoshida \cite{Nagai_Senba_Yoshida_1997} proved  existence, uniqueness and regularity of solutions under the condition  $\int_\Omega u_0(\x)\,\dx\in (0, 4\pi)$. In proving this, a variant of Moser--Tridinguer's inequality was used. In the particular case that  $\Omega$ be a ball, the above-mentioned condition becomes  $\int_\Omega u_0(\x)\,\dx\in (0,8\pi)$. Herrero and Velázquez \cite{Herrero_Velazquez_1997} dealt with the first blow--up framework by constructing some radially symmetric two-dimensional solutions which blow up within finite time. The next progress in this sense with $\Omega$ being non-radial and simply connected was the work of Horstmann and Wang \cite{Horstmann_Wang_2001} who found some unbounded solutions provided that $\int_\Omega u_0(\x)\,\dx > 4\pi$ and  $\int_\Omega u_0(\x)\,\dx\not\in \{4k\pi\, |\, k\in\N\}$. So far there is no supporting evidence as to whether such solutions may evolve to produce a blow-up phenomenon within finite time or whether, on the contrary, may increase to infinity with time. In three dimensions, Winkler \cite{Winkler_2013} proved that there exist radially symmetric solutions blowing up in finite time for any value of $\int_\Omega u_0(\x)\,\dx$.

The main tool \cite{Winkler_2013} in proving blow-up solutions is the energy law which stems from system \eqref{KS}.  Nevertheless, an inadequate approximation of lower bounds can trigger off oscillations of the variables, which can lead to spurious, blow-up solutions.

Concerning the numerical analysis for system \eqref{KS}, very little is said about numerical algorithms which keep lower bounds, are $L^1(\Omega)$-bounded and have a discrete energy law. Proper numerical  treatment of these properties is made difficult by the fact that the non-linearity occurs in the highest order derivative.  Numerical algorithms are mainly designed so as to keep lower bounds and to be mass-preserving. We refer the reader to \cite{Saito_2012, De-Leenheer_Gopalakrishnan_Zuhr_2013, Li_Shu_Yang_2017, Chertock_Epshteyn_Hu_Kurganov_2018, Sulman_Nguyen_2019}. We were pointed out by a referee the paper \cite{Guo_Li_Yang_2019}. In it, the authors rewrote system \eqref{KS} by using several \emph{ad hoc} auxiliary variables so that the resulting discontinuous Galerkin method fulfilled a discrete energy law, but no lower bounds were proved. As far as we are concerned, there is no numerical method facing lower bounds as well as a discrete energy law.

\subsection{Notation} We collect here as a reference some standard notation used throughout the paper.
For $p\in[1,\infty]$, we denote by $L^p(\Omega)$ the usual Lebesgue space, i.e.,
$$
L^p(\Omega) = \{v : \Omega \to \R\, :\, v \mbox{ Lebesgue-measurable}, \int_\Omega |v(\x)|^p {\rm d}\x<\infty \}.
$$
or
$$
L^\infty(\Omega) = \{v : \Omega \to \R\, :\, v \mbox{ Lebesgue-measurable}, {\rm ess}\sup_{\x\in \Omega} |v(\x)|<\infty \}.
$$

This space is a Banach space endowed with the norm
$\|v\|_{L^p(\Omega)}=(\int_{\Omega}|v(\x)|^p\,{\rm d}\x)^{1/p}$ if $p\in[1, \infty)$ or $\|v\|_{L^\infty(\Omega)}={\rm ess}\sup_{\x\in \Omega}|v(\x)|$ if $p=\infty$. In particular,  $L^2(\Omega)$ is a Hilbert space.  We shall use
$\left(u,v\right)=\int_{\Omega}u(\x)v(\x){\rm d}\x$ for its inner product and $\|\cdot\|$ for its norm.

Let $\alpha = (\alpha_1, \alpha_2)\in \mathds{N}^2$ be a multi-index with $|\alpha|=\alpha_1+\alpha_2$, and let
$\partial^\alpha$ be the differential operator such that
$$\partial^\alpha=
\Big(\frac{\partial}{\partial{x_1}}\Big)^{\alpha_1}\Big(\frac{\partial}{\partial{x_2}}\Big)^{\alpha_2}.$$

For $m\ge 0$ and $p\in[1, \infty)$, we consider $W^{m,p}(\Omega)$ to be the Sobolev space of all functions whose $m$ derivatives are in $L^p(\Omega)$, i.e.,
$$
W^{m,p}(\Omega) = \{v \in L^p(\Omega)\,:\, \partial^k v \in L^2(\Omega)\ \forall ~ |k|\le m\}
$$ associated to the norm
$$
\|f\|_{W^{m,p}(\Omega)}=\left(\sum_{|\alpha|\le m} \|\partial^\alpha f\|^p_{L^p(\Omega)}\right)^{1/p} \quad \hbox{for} \ 1 \leq p < \infty,
$$
and
$$
\|f\|_{W^{m,p}(\Omega)}=\max_{|\alpha|\le m} \|\partial^\alpha f\|_{L^\infty(\Omega)} \quad  \hbox{for} \  p = \infty.
$$
For $p=2$, we denote $W^{m,2}(\Omega)=H^m(\Omega)$. Moreover, we make of use the space
$$
H_N^2(\Omega)=\{ v\in H^2(\Omega) : \int_\Omega v(\x)\,\dx=0\mbox{ and } \partial_{\n} v=0 \mbox{ on } \partial\Omega\},
$$
for which is known that $\|v\|_{H^2_N(\Omega)}$ and $\|\Delta v\|$ are equivalent norms.
\subsection{Outline}
The remainder of this paper is organized in the following way.  In the next section we state our finite element space and some tools. In particular, we prove a discrete version of a variant of Moser--Trudinger's inequality.  In section 3, we apply our ideas to discretize system \eqref{KS} in space and time for defining our numerical method and formulate our main result. Next is section 4 dedicated to demonstrating lower bounds, a discrete energy law, and \emph{a priori} bounds all of which are local in time for approximate solutions. This is accomplished in a series of lemmas where the final argument is an induction procedure on the time step so as to obtain the above mentioned properties globally in time. Finally, in section 5, we consider two numerical examples regarding blow-up and non blow-up scenarios.

\section{Technical  preliminaries}
This section is mainly devoted to setting out the hypotheses and some auxiliary results concerning the finite element space that will use throughout this work.

\subsection{Hypotheses}
We construct the finite element approximation of \eqref{KS} under the following assumptions on the domain, the mesh, and the finite element space.
\begin{enumerate}
\item [(H1)] Let $\Omega$ be a convex, bounded domain of $\R^2$ with a polygonal boundary, and let $\theta_\Omega$ be the minimum interior angle at the vertices of $\partial\Omega$.
\item[(H2)] Let $\{{\mathcal T}_{h}\}_{h>0}$  be a family of acute, shape-regular, quasi-uniform triangulations of  $\overline{\Omega}$ made up of triangles, so that $\overline \Omega=\cup_{T\in {\mathcal T}_h}T$, where $h=\max_{T\in \mathcal{T}_h} h_T$, with $h_T$ being the diameter of $T$. More precisely, we assume that
\begin{enumerate}
\item[(a)]  there exists $\alpha>0$, independent of $h$, such that
$$
\min\{  {\rm diam}\,  B_T\, : \, T\in\mathcal{T}_h \}\ge \alpha h,
$$
where $B_T$ is the largest ball contained in $T$, and
\item[(b)] there exists $\beta > 0$ such that  every angle between two edges of a triangle $T$ is bounded by $\frac{\pi}{2} - \beta$.
\end{enumerate}
Further, let ${\mathcal N}_h = \{\a_i\}_{i\in I}$ be the coordinates of the nodes of ${\mathcal T}_h$.
\item [(H3)]  Associated with  ${\mathcal T}_h$ is the finite element space
$$
X_h = \left\{ x_h \in {C}^0(\overline\Omega) \;:\;
x_h|_T \in \mathcal{P}_1(T), \  \forall T \in \mathcal{T}_h \right\},
$$
where $\mathcal{P}_1(T)$ is the set of linear polynomials on  $T$.  Let $\{\varphi_\a\}_{\a\in{\mathcal{N}_h}}$ be the standard basis functions for $X_h$.
\end{enumerate}
\subsection{Auxiliary results} Our first result is concerned with the sign of the entries of the rigid matrix.
\begin{proposition} Let $\Omega$ be a polygonal. Consider $X_h$ to be constructed over $\mathcal{T}_h$ being acute. Then, for each $T\in\mathcal{T}_h$ with vertices $\{\boldsymbol{a}_1,\boldsymbol{a}_2,\boldsymbol{a}_3\}$,  there exists a constant $C_{\rm neg}>0$, depending on $\beta$, but otherwise independent of $h$ and $T$, such that
\begin{equation}\label{off-diagonal}
\int_T \nabla\varphi_{\a_i}\cdot\nabla\varphi_{\a_j}\dx  \le - C_{\rm neg}
\end{equation} for all $\a_i,\a_j\in T$ with $i\not=j$, and
\begin{equation}\label{diagonal}
\int_T\nabla\varphi_{\a_i}\cdot\nabla\varphi_{\a_i}\dx\ge C_{\rm neg}
\end{equation}
for all $\a_i\in T$.
\end{proposition}

A proof of \eqref{off-diagonal} and \eqref{diagonal} can be found in \cite{JVGS_2018}. 
%

Both local and global finite element properties for $X_h$ will be needed such as inverse estimates and bounds for the interpolation error. We first recall some local inverse estimates. See \cite[Lem. 4.5.3]{Brenner_Scott_2008} or \cite[Lem. 1.138]{Ern_Guermond_2004} for a proof.
\begin{proposition} Let $\Omega$ be polygonal. Consider $X_h$ to be constructed over $\mathcal{T}_h$ being quasi-uniform. Then, for each $T\in \mathcal{T}_h$ and  $p\in[2,\infty]$, there exists a constant $C_{\rm inv}>0$, independent of $h$ and $T$, such that, for all $x_h\in X_h$,
\begin{equation}\label{inv_W1pToLp_local}
\|\nabla x_h \|_{L^p(T)}\le C_{\rm inv}\, h^{-1} \|x_h\|_{L^p(T)}
\end{equation}
and
\begin{equation}\label{inv_W1infToW1p_local}
\|\nabla x_h \|_{L^\infty(T)}\le C_{\rm inv}\, h^{-\frac{2}{p}} \|\nabla x_h\|_{L^p(T)}.
\end{equation}
\end{proposition}
Concerning global inverse inequalities, we need the following.
\begin{proposition} Let $\Omega$ be polygonal. Consider $X_h$ to be constructed over $\mathcal{T}_h$ being quasi-uniform. Then for each $p\in[2,\infty]$,  there exists a constant $C_{\rm inv}>0$, independent of $h$, such that, for all $x_h\in X_h$,
\begin{equation}\label{inv_LinfToL2_global}
\|x_h\|_{L^\infty(\Omega)}\le C_{\rm inv}\, h^{-1} \|x_h\|,
\end{equation}
\begin{equation}\label{inv_W1pToLp_global}
\|\nabla x_h \|_{L^p(\Omega)}\le C_{\rm inv}\, h^{-1} \|x_h\|_{L^p(\Omega)},
\end{equation}
\begin{equation}\label{inv_W1pToH1_global}
\|\nabla x_h \|_{L^p(\Omega)}\le C_{\rm inv}\, h^{-2(\frac{1}{2}-\frac{1}{p})} \|\nabla x_h\|_{L^2(\Omega)},
\end{equation}
and
\begin{equation}\label{inv_W1infToW1p_global}
\|\nabla x_h \|_{L^\infty(\Omega)}\le C_{\rm inv}\, h^{-\frac{2}{p}} \|\nabla x_h\|_{L^p(\Omega)}.
\end{equation}
\end{proposition}

We introduce $\mathcal{I}_h : C(\Omega) \to  X_h$, the standard nodal interpolation operator, such that $\mathcal{I}_h \eta(\a)=\eta(\a)$ for all $\a\in\mathcal{N}_h$. Associated with $\mathcal{I}_h$, a discrete inner product on $X_h$ is defined by
$$
(x_h,\bar x_h)_h=\int_\Omega \mathcal{I}_h(x_h(\x)\bar x_h(\x))\, \dx.
$$
We also introduce
$$
\|x_h\|_h=(x_h,x_h)_h^{\frac{1}{2}}.
$$
Local and global error bounds for $\mathcal{I}_h$ are as follows (c.f.  \cite[Thm. 4.4.4]{Brenner_Scott_2008} or \cite[Thm. 1.103]{Ern_Guermond_2004} for a proof).
\begin{proposition} Let $\Omega$ be polygonal. Consider $X_h$ to be constructed over $\mathcal{T}_h$ being quasi-uniform. Then, for each $T\in \mathcal{T}_h$, there exists $C_{\rm app}>0$, independent of $h$ and $T$, such that
\begin{equation}\label{interp_error_nodal_L1_and_W21_local}
\|\varphi-\mathcal{I}_h\varphi\|_{L^1(T)}\le C_{\rm app} h^2 \|\nabla^2\varphi\|_{L^1(T)} \quad\forall\, \varphi\in W^{2,1}(T).
\end{equation}
\end{proposition}
\begin{proposition} Let $\Omega$ be polygonal. Consider $X_h$ to be constructed over $\mathcal{T}_h$ being quasi-uniform. Then it follows that  there exists $C_{\rm app}>0$, independent of $h$, such that
\begin{equation}\label{interp_error_nodal_H1_and_H2_global}
\|\nabla(\varphi-\mathcal{I}_h\varphi)\|_{L^2(\Omega)} \le C_{\rm  app} h \|\nabla^2\varphi\|_{L^2(\Omega)}\quad\forall \varphi\in H^2(\Omega).
\end{equation}
\end{proposition}


\begin{corollary} Let $\Omega$ be polygonal. Consider $X_h$ to be constructed over $\mathcal{T}_h$ being quasi-uniform.  Let $n\in \mathds{N}$. Then it follows that there exist three positive constants $C_{\rm app}$, $C_{\rm com}$, and $C_{\rm sta}$, independent of $h$, such that
\begin{equation}\label{Error:I_h(x_h^n)-x_h^n_L1}
\|x_h^n-\mathcal{I}_h(x_h^n)\|_{L^1(\Omega)}\le
C_{\rm app} n(n-1)h^2 \int_\Omega |x_h(\x)|^{n-2} |\nabla x_h(\x)|^2\, \dx,
\end{equation}
\begin{equation}\label{Error: I_h(xh_bxar)-xh_bxh_L1}
\|x_h \overline x_h-\mathcal{I}_h(x_h \overline x_h)\|_{L^1(\Omega)}\le C_{\rm com} h\, \|x_h\|_{L^2(\Omega)} \, \|\nabla \overline x_h\|_{L^2(\Omega)}
\end{equation}
and
\begin{equation}\label{Stab:I_h_Ln}
\|x_h^n\|_{L^1(\Omega)}\le \|\mathcal{I}_h(x^n_h)\|_{L^1(\Omega)}\le C_{\rm sta} \|x_h^n\|_{L^1(\Omega)},
\end{equation}
where $C_{\rm sta}$ depends on $n$.
\end{corollary}
\begin{proof}
Let $T\in\mathcal{T}_h$ and compute
$$
\nabla^2(x_h^n)=n (n-1) x_h^{n-2}\sum_{i,j=1}^d  \partial_i x_h \partial_j x_h\quad\mbox{ on }\quad T.
$$
Then, from \eqref{interp_error_nodal_L1_and_W21_local} and the above identity, we have
$$
\begin{array}{rcl}
\|x_h^n-\mathcal{I}_h(x_h^n)\|_{L^1(T)}
&\le& C_{\rm app} h^2_K \|\nabla^2(x_h^n)\|_{L^1(T)}
\\
&\le&\displaystyle
C_{\rm app} n(n-1)h^2 \int_T |x_h(\x)|^{n-2} |\nabla x_h(\x)|^2\, \dx.
\end{array}
$$
Summing over $T\in\mathcal{T}_h$ yields \eqref{Error:I_h(x_h^n)-x_h^n_L1}. The proof of \eqref{Error: I_h(xh_bxar)-xh_bxh_L1} follows very closely the arguments of \eqref{Error:I_h(x_h^n)-x_h^n_L1} for $n=2$.  The first part of assertion \eqref{Stab:I_h_Ln} is a simple application of Jensen's inequality, whereas the second part follows from \eqref{Error:I_h(x_h^n)-x_h^n_L1} on using Hölder's inequality, \eqref{inv_W1pToLp_global} for $p=n$ and, later on, reverse Minkowski's inequality.
\end{proof}
The proof of the following proposition can be found in \cite{Chang_Yang_1988}. It is a generalization of a Moser--Trudinger-type inequality.
\begin{proposition}[Moser-Trudinger] Let $\Omega$ be polygonal with $\theta_\Omega$ being the minimum interior angle at the vertices of $\Omega$.  Then there exists a constant $C_\Omega>0$ depending on $\Omega$ such that for all $u\in H^1(\Omega)$ with $\|\nabla u\|\le 1$ and $\int_\Omega u(\x)\,\dx=1$, it follows that
\begin{equation}\label{Moser-Trudinger}
\int_\Omega e^{\alpha |u(\x)|^2} \dx\le C_\Omega,
\end{equation}
where $\alpha\le 2 \theta_\Omega$.
\end{proposition}

\begin{corollary} Let $\Omega$ be polygonal with $\theta_\Omega$ being the minimum interior angle at the vertices of $\partial\Omega$. Consider $X_h$ to be constructed over $\mathcal{T}_h$ being quasi-uniform.  Let $u_h\in X_h$ with $u_h>0$. Then it follows that there exists a constant $C_{\rm MT}>0$, independent of $h$, such that
\begin{equation}\label{Moser-Trudinger-Ih}
\int_\Omega \mathcal{I}_h (e^{u_h(\x)})\,\dx\le C_\Omega (1+C_{\rm MT} \|\nabla u_h\|^2) e^{\displaystyle\frac{1}{8\theta_\Omega} \|\nabla u_h\|^2+\frac{1}{|\Omega|}\|u_h\|_{L^1(\Omega)}}.
\end{equation}
\end{corollary}
\begin{proof} From \eqref{Error:I_h(x_h^n)-x_h^n_L1}, we have
\begin{equation}\label{co2.8-lab1}
\begin{array}{rcl}
\displaystyle
\int_\Omega \mathcal{I}_h (e^{u_h(\x)})\,\dx&=&\displaystyle
\int_\Omega (1+u_h(\x))\,\dx+\sum_{n=2}^\infty\frac{1}{n!} \int_\Omega\mathcal{I}_h(u_h^n(\x))\dx
\\
&\le&\displaystyle
\sum_{n=0}^\infty\frac{1}{n!} \int_\Omega u^n_h(\x)\,\dx
 \\
&&\displaystyle +\sum_{n=2}^\infty\frac{C_{\rm app} n(n-1) h^2}{n!}  \int_\Omega |\nabla u_h(\x)|^2 u_h^{n-2}(\x)\,\dx
\\
&=&\displaystyle
\int_\Omega (1+ C_{\rm app} h^2 |\nabla u_h(\x)|^2) e^{u_h(\x)}\,\dx.
\end{array}
\end{equation}
Let $\displaystyle v_h=\frac{u_h-m}{\|\nabla u_h\|}$ with $\displaystyle m=\frac{1}{|\Omega|}\int_\Omega u_h(\x) \,\dx$. Young's inequality gives
\begin{equation}\label{co2.8-lab2}
u_h=\|\nabla u_h\|v_h+ m\le \frac{1}{8 \theta_\Omega} \|\nabla u_h\|^2+ 2 \theta_\Omega |v_h|^2+ m.
\end{equation}
Thus, combining \eqref{co2.8-lab1} and \eqref{co2.8-lab2} yields, on noting  \eqref{inv_W1infToW1p_global} for $p=2$ and \eqref{Moser-Trudinger}, that
$$
\begin{array}{rcl}
\displaystyle
\int_\Omega \mathcal{I}_h (e^{u_h(\x)})\,\dx&\le&\displaystyle e^{\frac{1}{8 \theta_\Omega} \|\nabla u_h\|^2 +m}
\int_\Omega (1+ C_{\rm app} h^2 \|\nabla u_h(\x)\|^2 |\nabla v_h(\x)|^2) e^{2 \theta_\Omega |v_h(\x)|^2}\,\dx
\\
&\le&\displaystyle e^{\frac{1}{8 \theta_\Omega} \|\nabla u_h\|^2 +m}
(1+ C_{\rm app} C_{\rm inv} \|\nabla u_h(\x)\|^2 ) \int_\Omega e^{2 \theta_\Omega |v_h(\x)|^2}\,\dx
\\
&\le&\displaystyle C_\Omega(1+ C_{\rm app} C_{\rm inv} \|\nabla u_h(\x)\|^2 ) e^{\frac{1}{8 \theta_\Omega} \|\nabla u_h\|^2 +m}.
\end{array}
$$

\end{proof}
An (average) interpolation operator into $X_h$ will be required in order to properly initialize our numerical method.  We refer to \cite{Scott_Zhang_1990, Girault_Lions_2001}.
\begin{proposition} Let $\Omega$ be polygonal. Consider $X_h$ to be constructed over $\mathcal{T}_h$ being quasi-uniform. Then there exists an (average) interpolation operator  $\mathcal{Q}_h$ from $L^1(\Omega)$ to $X_h$ such that
\begin{equation}\label{Q-sta}
\|\mathcal{Q}_h \psi\|_{W^{s,p}(\Omega)}\le C_{\rm sta} \| \psi \|_{W^{s,p}(\Omega)}\quad \mbox{for }  s=0,1\mbox{ and } 1\le p\le\infty,
\end{equation}
and
\begin{equation}\label{Q-app}
\|\mathcal{Q}_h(\psi)- \psi \|_{W^{s,p} (\Omega)}\le C_{\rm app}  h^{1+m-s} \|\psi \|_{W^{m+1, p}(\Omega)} \quad \mbox{for } 0\le s\le m\le 1.
\end{equation}
\end{proposition}

Moreover, let $-\tilde\Delta_h$ be defined from $X_h$ to $X_h$ as
\begin{equation}\label{Discrete-Laplacian}
-(\tilde\Delta_h x_h, \bar x_h)_h=(\nabla x_h, \nabla \bar x_h)\quad \mbox{ for all } \bar x_h\in X_h,
\end{equation}
and let $x(h)\in H^2_N(\Omega)$ be such that
\begin{equation}\label{eq:x(h)}
\left\{
\begin{array}{rclcl}
-\Delta x(h)&=&-\tilde\Delta_h x_h&\mbox{ in }&\Omega,
\\
\partial_{\n}x(h)&=&0&\mbox{ on }&\partial\Omega.
\end{array}
\right.
\end{equation}
From elliptic regularity theory, the well-posedness of \eqref{eq:x(h)} is  ensured by the convexity assumption stated in $\rm (H1)$ and
\begin{equation}\label{continuity_x(h)}
\|x(h)\|_{H^2_N(\Omega)}\le C \|-\widetilde\Delta_h x_h\|.
\end{equation} See \cite{Grisvard_1985} for a proof.
\begin{proposition} Let $\Omega$ be a convex polygon. Consider $X_h$ to be constructed over $\mathcal{T}_h$ being quasi-uniform. Then there exists a constant $C_{\rm Lap}>0$, independent of $h$, such that
\begin{equation}\label{Error-x(h)}
\|\nabla(x(h)-x_h)\|_{L^2(\Omega)}\le C_{\rm Lap} h \|\tilde\Delta_h x_h\|_{L^2(\Omega)}.
\end{equation}
\end{proposition}
\begin{proof} We refer the reader to \cite{JVGS_2018} for a proof which uses  \eqref{interp_error_nodal_H1_and_H2_global} and  \eqref{Error: I_h(xh_bxar)-xh_bxh_L1}.
\end{proof}
\begin{corollary} Let $\Omega$ be a convex polygon. Consider $X_h$ to be constructed over $\mathcal{T}_h$ being quasi-uniform. Then, for each $p\in[2,\infty]$, there exists a constant $C_{\rm sta}>0$, independent of $h$, such that
\begin{equation}\label{stab:Discrete-Laplacian-W1p-H2-xh}
 \|\nabla x_h\|_{L^p(\Omega)} \le C_{\rm sta} \|-\widetilde\Delta_h x_h\|.
\end{equation}
\end{corollary}
\begin{proof}
The triangle inequality gives
$$
\|\nabla x_h\|_{L^p(\Omega)}\le \|\nabla x_h-\nabla\mathcal{Q}_h x(h)\|_{L^p(\Omega)}+\|\nabla\mathcal{Q}_h x(h)\|_{L^p(\Omega)}
$$
and hence applying \eqref{inv_W1pToH1_global},  \eqref{Error-x(h)}, \eqref{continuity_x(h)}, \eqref{Q-app}, \eqref{Q-sta}, and Sobolev's inequality yields \eqref{stab:Discrete-Laplacian-W1p-H2-xh}.
\end{proof}

\section{Presentation of main result}
We now define our numerical approximation of system \eqref{KS}. Assume that $(u_0, v_0)\in H^1(\Omega)\times H^2(\Omega) $ with $u_0>0$ and $v_0\ge 0$ a. e. in $\Omega$.

We begin by approximating the initial data $(u_0, v_0)$ by $(u^0_h,v^0_h)\in X_h^2$ as follows. Define
\begin{equation}\label{def:u^0_h}
u^0_h=\mathcal{Q}_h u_0,
\end{equation}
which satisfies
\begin{equation}\label{cond:u^0_h}
u_h^0>0\mbox{ a. e. in }\Omega,\quad \|u^0_h\|_{L^1(\Omega)}\le C_{\rm sta}\|u_0\|_{L^1(\Omega)}\quad\mbox{ and }\quad \|u^0_h\|\le C_{\rm sta} \|u_0\|
\end{equation}
and
\begin{equation}\label{def:v^0_h}
v^0_h=\mathcal{Q}_h v_0,
\end{equation}
which satisfies
\begin{equation}\label{cond:v^0_h}
v_h^0\ge0\mbox{ a. e. in }\Omega,\quad \|v^0_h\|_{H^1(\Omega)}\le C  \|v_0\|_{H^1(\Omega)}\quad\mbox { and }\quad \|\widetilde\Delta_h v_h^0\|\le C \|\Delta v_0\|.
\end{equation}

Given $N\in\mathds{N}$, we let $0 = t_0 < t_1 < ... < t_{N-1} < t_N = T$ be a uniform partitioning of [0,T] with  time step $k=\frac{T}{N}$. To simplify the notation we define the time-increment operator $\delta_t\phi^{n+1}_h=\frac{\phi^{n+1}_h-\phi^n_h}{k}$.

Known $(u_h^n,v^n_h)\in X_h\times X_h$, find $(u_h^{n+1}, v_h^{n+1})\in X_h\times X_h$ such that
\begin{equation}\label{eq:u_h}
  (\delta_t u^{n+1}_h, x_h)_h+(\nabla u_h^{n+1},\nabla x_h)=(\nabla v_h^n, u_h^{n+1}\nabla x_h)
\end{equation}
and
\begin{equation}\label{eq:v_h}
(\delta_t v^{n+1}_h, x_h)_h +(\nabla v^{n+1}_h, \nabla x_h)+(v_h^{n+1}, x_h)_h=(u^{n+1}_h,x_h)_h
\end{equation}
 for all $x_h \in X_h$.

It should be noted that scheme \eqref{eq:u_h}-\eqref{eq:v_h} combines a finite element method together a mass-lumping technique to treat some terms and a semi-implicit time integrator. The resulting scheme is linear and decouples the computation of $u^{n+1}_h$ and $v^{n+1}_h$.

In order to carry out our numerical analysis we must rewrite the chemotaxis term by using a barycentric quadrature rule as follows. Let $T\in\mathcal{T}_h$ and consider $\boldsymbol{b}_T\in T $ to be the barycenter of $T$. Then let $\overline{u}^{n+1}_h $ be the interpolation of $u^{n+1}_h$ into $X_h^0$, with $X_h^0$ being
the space of all piecewise constant functions over $\mathcal{T}_h$, defined by
\begin{equation}\label{Interp_u}
\overline{u}_h^{n+1}|_T=u^{n+1}_h(\boldsymbol{b}_{T}).
\end{equation}
As a result, one has
\begin{equation}\label{eq:u_h-equiv}
(\nabla v_h^n, u_h^{n+1}\nabla x_h)=\sum_{T\in K} |T| \nabla v_h^n\cdot \nabla x_h u_h^{n+1}(\boldsymbol{b}_T)=(\nabla v_h^n, \overline{u}_h^{n+1}\nabla x_h).
\end{equation}

Let us define
\begin{equation}\label{def:energy}
\mathcal{E}_0(u_h, v_h)=\frac{1}{2}\|v_h\|^2_h+\frac{1}{2}\|\nabla v_h\|^2-(u_h, v_h)_h+(\log u_h, u_h)_h,
\end{equation}
\begin{equation}\label{def:energy_II}
\mathcal{E}_1(u_h, v_h)=\|\nabla u_h\|_h^2+\|\nabla\widetilde\Delta_h v_h\|^2.
\end{equation}
and, for each $\varepsilon,\delta\in (0,1)$,
\begin{equation}\label{def:R_0}
\mathcal{R}_0^{\varepsilon, \delta}(u^0_h,v_h^0):= \frac{1}{\delta e}+\frac{\|u_h^0\|_{L^1(\Omega)}}{\delta}\left(\frac{C_\Omega}{\varepsilon} +\varepsilon+\frac{(1+\delta)}{|\Omega|} (\|v^0_h\|_{L^1(\Omega)}+\|u^{0}_h\|_{L^1(\Omega)})\right).
\end{equation}
Associated with the above definitions, consider
$$
\mathcal{B}_0(u_h,v_h)=\frac{1}{\delta}\mathcal{E}_0(u_h, v_h)+\mathcal{R}^{\delta,\varepsilon}_0(u_h,v_h),
$$
$$
\mathcal{B}_1(u_h,v_h)=(1+\frac{1}{\delta})\mathcal{E}_0(u_h, v_h)+\mathcal{R}^{\delta,\varepsilon}_0(u_h,v_h)+2\frac{|\Omega|}{e},
$$
and
$$
\mathcal{B}_2(u_h,v_h)=\mathcal{E}_0(u_h,v_h)+\mathcal{B}_0(u_h,v_h)+\mathcal{B}_1(u_h,v_h).
$$
Finally, define
$$
\mathcal{F}(u_h, v_h)=e^{\mathcal{B}_2(u_h,v_h)+T^{\frac{1}{2}}\mathcal{B}^{\frac{1}{2}}_2(u_h,v_h)} (\mathcal{E}_0(u_h, v_h) +C T\mathcal{B}_1^3(u_h, v_h)+C T \|u_h\|_{L^1(\Omega)}),
$$

The definition of the above quantities will be apparent later.

We are now prepared to state the main result of this paper.

\begin{theorem}\label{Th:main} Assume that hypotheses $\rm (H1)$--$\rm(H3)$ are satisfied. Let $(u_0,v_0)\in H^1(\Omega)\times H^2(\Omega)$ with $u_0>0$ such that $\|u_0\|_{L^1(\Omega)}\in (0,4 \theta_\Omega)$ and $v_0\ge0$, and take $u^0_h>0$ and $v^0_h\ge0$ defined by \eqref{def:u^0_h} and \eqref{def:v^0_h}, respectively. Assume that $(h,k)$ fulfill
\begin{equation}\label{restriction-h-k}
 \frac{k}{h^2} \mathcal{F}(u_h^0,v_h^0)<  \frac{1}{2C}
\end{equation}
and
\begin{equation}\label{restriction-h}
 h^{1-\frac{2}{p}} \mathcal{F}^{\frac{1}{2}}(u_h^0,v_h^0)<\frac{C_{\rm neg}}{C}.
\end{equation}
Then the sequence $\{(u^m_h, v^m_h)\}_{m=0}^N$ computed via \eqref{eq:u_h} and \eqref{eq:v_h}  satisfies the following properties, for all $m\in\{0,\cdots, N\}$:
\begin{itemize}
\item Lower bounds:
\begin{equation}\label{Global-Lower-Bound-uh}
u^{m}_h(\x)>0
\end{equation}
and
\begin{equation}\label{Global-Lower-Bound-vh}
v^{m}_h(\x)\ge 0
\end{equation}
for all $\x\in\Omega$,
\item  $L^1(\Omega)$-bounds:
\begin{equation}\label{Global-L1-Bound-uh}
\|u^m_h\|_{L^1(\Omega)}=\|u^0_h\|_{L^1(\Omega)}
\end{equation}
and
\begin{equation}\label{Global-L1-Bound-vh}
\|v^m_h\|_{L^1(\Omega)}\le \|v^0_h\|_{L^1(\Omega)}+\|u^0_h\|_{L^1(\Omega)}.
\end{equation}
\item A discrete energy law:
\begin{equation}\label{Global-Energy-Law}
\begin{array}{rcl}
\displaystyle
\mathcal{E}_0(u^{m}_h,v^{m}_h)+k\sum_{r=1}^m(\|\delta_t v^{r}_h\|^2_h
\displaystyle
+k\|\mathcal{A}^{-\frac{1}{2}}_h(u^{r}_h)\nabla u^{r}_h-\mathcal{A}^{\frac{1}{2}}_h(u^{r}_h)\nabla v^{r-1}_h\|^2)\le\mathcal{E}_0(u^0_h,v^0_h),
\end{array}
\end{equation}
where $\mathcal{A}_h$ is defined in \eqref{def-A}.
\end{itemize}
Moreover, if we are given $h$  such that
\begin{equation}\label{restriction-h_II}
C h^{1-\frac{2}{p}} \mathcal{E}_1(u^0_h,v^0_h)\le \frac{5}{12},
\end{equation}
it follows that
\begin{equation}\label{Global-Energy-Bound}
\begin{array}{r}
\displaystyle
\mathcal{E}_1(u^{m}_h, v^{m}_h)+\frac{k}{2} \sum_{r=1}^m ( \|\nabla u^{r}_h\|^2+\|\nabla\widetilde\Delta_h v^{r}_h\|^2)\le \mathcal{F}(u^0_h, v^0_h).
\end{array}
\end{equation}
\end{theorem}

\begin{remark} 

The constant $C$ is not easy to compute in practice. Hence \eqref{restriction-h-k} and \eqref{restriction-h} should only be seen as theoretical conditions, meaning that $k/h^2$ and $h^{2-\frac{1}{p}}$ have to be sufficiently small to reach \eqref{Global-Lower-Bound-uh} and \eqref{Global-Lower-Bound-vh} on $(0,T]$.
\end{remark}

As system \eqref{eq:u_h}-\eqref{eq:v_h} is linear, existence follows from uniqueness. The latter is an immediate outcome of  \emph{a priori} bounds for $ (u^{n+1}_h, v^{n+1}_h)$.

\section{Proof of main result}

In this section we address the proof of Theorem \ref{Th:main}. Rather than prove \emph{en masse} the estimates in Theorem \ref{Th:main}, because all of them are connected, we have divided the proof into various subsections for the sake of clarity. The final argument will be an induction procedure on $n$ relied on the semi-explicit time discretization employed in \eqref{eq:u_h}.
\subsection{Lower bounds and a discrete energy law} We first demonstrate lower bounds for $(u^{n+1}_h, v^{n+1}_h)$ and, as a consequence of this, a discrete local-in-time energy law is established.
\begin{lemma}[Lower bounds]\label{lm:Lower_Bounds} Assume that $\rm (H1)$--$\rm(H3)$ are satisfied. Let $u^n_h>0$ and $v^n_h\ge0$ and let
\begin{equation}\label{hyp:induction}
 \|\widetilde\Delta_h v^n_h\|^2\le \mathcal{F}(u_h^0,v_h^0).
\end{equation}
Then if one chooses $(h,k)$ satisfying \eqref{restriction-h-k} and \eqref{restriction-h}, it follows that the solution $(u^{n+1}_h, v^{n+1}_h)\in X_h^2$ computed via \eqref{eq:u_h} and \eqref{eq:v_h} are lower bounded, i.e, for all $\x\in\Omega$,
\begin{equation}\label{positivity-uh}
u^{n+1}_h(\x)>0
\end{equation}
and
\begin{equation}\label{positivity-vh}
v^{n+1}_h(\x)\ge0.
\end{equation}
\end{lemma}
\begin{proof} Since $u_h^{n+1}$ and $v_h^{n+1}$ are piecewise linear polynomial functions, it will suffice to prove that \eqref{positivity-uh} and \eqref{positivity-vh} hold at the nodes. To do this, let $T\in\mathcal{T}_h$ be a fixed triangle with vertices $\{\a_1, \a_2,\a_3\}$, and choose two of them, i.e. $\a_i,\a_j\in T$ with $i\not=j$. Then, from \eqref{inv_W1pToLp_local} for $p=\infty$, \eqref{inv_W1infToW1p_local}, \eqref{stab:Discrete-Laplacian-W1p-H2-xh}, and \eqref{hyp:induction}, we have on noting \eqref{eq:u_h-equiv}   that
\begin{equation}\label{lm5.1-lab1}
\begin{array}{rcl}
\displaystyle
\int_T \overline{\varphi}_{\a_i}\nabla v^n_h\cdot \nabla \varphi_{\a_j}\, \dx&=&\displaystyle
\int_T\varphi_{\a_i}(\boldsymbol{b}_T)\nabla v_h^n\cdot \nabla\varphi_{\a_j}\dx
\\
&\le&|T| \|\varphi_{\a_i}\|_{L^\infty(T)} \|\nabla v_h^n\|_{L^\infty(T)} \|\nabla\varphi_{\a_j}\|_{L^\infty(T)}
\\
&\le& C h \|\nabla v^n_h\|_{L^{\infty}(T)}\le C h^{1-\frac{2}{p}} \|\nabla v^n_h\|_{L^p(T)}
\\
&\le& C h^{1-\frac{2}{p}} \|\widetilde\Delta_h v^n_h\|_{L^2(\Omega)}\le C h^{1-\frac{2}{p}} \mathcal{F}^{\frac{1}{2}}(u_h^0,v_h^0).
\end{array}
\end{equation}
If we now compare \eqref{off-diagonal} with \eqref{lm5.1-lab1}, we find on recalling \eqref{restriction-h}  that
\begin{align*}
\int_T\nabla \varphi_{\a_i}\cdot\nabla \varphi_{\a_j}\dx &- \int_T\overline{\varphi}_{\a_i} \nabla v_h^n\cdot\nabla\varphi_{\a_j}\dx
\\
&\le -C_{\rm neg} + C h^{1-\frac{2}{p}} \mathcal{F}^{\frac{1}{2}}(u_h^0,v_h^0)<0
\end{align*}
and on summing over $T\in {\rm supp }\, \varphi_{\a_i}\cap {\rm supp }\, \varphi_{\a_j}  $ that
\begin{equation}\label{lm5.1-lab2}
(\nabla \varphi_{\a_i}, \nabla \varphi_{\a_j}) - (\overline{\varphi}_{\a_i}\nabla v^n_h,\nabla\varphi_{\a_j})<0.
\end{equation}
Analogously, we have, from \eqref{diagonal}, that
\begin{equation}\label{lm5.1-lab3}
(\nabla \varphi_{\a_i}, \nabla \varphi_{\a_i}) - (\overline{\varphi}_{\a_i} \nabla v^n_h,\nabla\varphi_{\a_i})>0
\end{equation}
holds under assumption \eqref{restriction-h}.

Let $u_{h}^{\rm min}\in X_h$ be defined as
$$
u_{h}^{\rm min}=\sum_{\a\in\mathcal{N}_h} u_{h}^{-}(\a)\varphi_\a,
$$
where $u_{h}^{-}(\a)=\min\{0, u^{n+1}_{h}(\a)\}$. Analogously, one defines $u^{\rm max}_h\in
X_h$ as
$$
u_{h}^{\rm max}=\sum_{\a\in\mathcal{N}_h} u_{h}^{+}(\a)\varphi_\a,$$
where $u_{h}^{+}(\a)=\max\{0, u^{n+1}_{h}(\a)\}$. It is easy to check that  $u^{n+1}_{h}=u_{h}^{\rm min}+ u_{h}^{\rm max}$.

Set $\bar x_h= u^{\rm min}_h$ in \eqref{eq:u_h} using \eqref{eq:u_h-equiv} to get
\begin{equation}\label{lm5.1-lab4}
(\delta_t u^{n+1}_h, u^{\rm min}_h)_h+ (\nabla u^{n+1}_h,\nabla u^{\rm min}_h)
-(\overline{u}^{n+1}_h\nabla v^n_h,\nabla u^{\rm min}_h)=0.
\end{equation}
Our goal is to show that $u_h^{\rm min}\equiv0$. Indeed, note that
$$
(u^{n+1}_h, u^{\rm min}_h)_h=(u^{\rm min}_h+u^{\rm max}_h, u^{\rm min}_h)_h=\|u^{\rm min}_h\|_h^2.
$$
and hence
\begin{equation}\label{lm5.1-lab6}
(\delta_t u^{n+1}_h, u^{\rm min}_h)_h= \frac{1}{k}\| u_h^{\rm min}\|_h^2-\frac{1}{k}(u^n_h, u^{\rm min}_h)\ge \|u_h^{\rm min}\|_h^2,
\end{equation}
where we have used the fact that $u^n_h>0$. One can further write
\begin{align*}
(\nabla u^{n+1}_h,\nabla u^{\rm min}_h)&-(\overline{u}^{n+1}_h\nabla v^n_h,\nabla u^{\rm min}_h )
\\
=&(\nabla u^{\rm max}_h,\nabla u^{\rm min}_h)-(\overline{u}^{\rm max}_h\nabla v^n_h,\nabla u^{\rm min}_h )
\\
&+(\nabla u^{\rm min}_h,\nabla u^{\rm min}_h)-(\overline{u}^{\rm min}_h\nabla v^n_h,\nabla u^{\rm min}_h),
\end{align*}
whereupon we deduce from \eqref{lm5.1-lab2} and \eqref{lm5.1-lab3}  that
\begin{align*}
\nonumber
(\nabla u^{\rm max}_h,&\nabla u^{\rm min}_h)-( \overline{u}^{\rm max}_h\nabla v^n_h,\nabla u^{\rm min}_h )
\\
=&\displaystyle
\sum_{\a\not=\tilde\a\in\mathcal{N}_h} u^{\rm max}_h(\a) u^{\rm min}_h(\tilde\a)\Big[(\nabla\varphi_\a,\nabla\varphi_{\tilde \a} )-(\overline{\varphi}_\a\nabla v^n_h,\nabla \varphi_{\tilde \a})\Big]
\\\nonumber&\displaystyle
+\sum_{\a\in\mathcal{N}_h} u^{\rm max}_h(\a) u^{\rm min}_h(\a)\Big[ (\nabla\varphi_\a,\nabla\varphi_{\a})-(\overline{\varphi}_\a\nabla v^n_h,\nabla \varphi_{\a})\Big]\ge0,
\end{align*}
since $u^{\rm max}_h(\a) u^{\rm min}_h(\tilde\a)\le0$ and $u^{\rm max}_h(\a) u^{\rm min}_h(\a)=0$. Therefore,
\begin{equation}\label{lm5.1-lab8}
(\nabla u^{\rm min}_h,\nabla u^{\rm min}_h)-(\overline{u}^{\rm min}_h\nabla v^n_h,\nabla u^{\rm min}_h )\le(\nabla u^{n+1}_h,\nabla u^{\rm min}_h)
-( \overline{u}^{n+1}_h\nabla v^n_h,\nabla u^{\rm min}_h )
\end{equation}
As a result, we infer on applying \eqref{lm5.1-lab6} and \eqref{lm5.1-lab8}  into \eqref{lm5.1-lab4}  that
$$
\|u^{\rm min}_h\|^2_h+k \|\nabla u^{\rm min}_h\|^2\le k (\overline{u}^{\rm min}_h\nabla v^n_h,\nabla u^{\rm min}_h).
$$
But inequalities \eqref{inv_LinfToL2_global}, \eqref{Stab:I_h_Ln} for $n=2$, \eqref{stab:Discrete-Laplacian-W1p-H2-xh} for $p=2$, \eqref{hyp:induction}, and \eqref{restriction-h-k} allow us to estimate
$$
\begin{array}{rcl}
\| u^{\rm min}_h\|^2_h+k \|\nabla u^{\rm min}_h\|^2&\le&  k \|u^{\rm min}_h\|_{L^\infty(\Omega)} \|\nabla v^n_h\|_{L^2(\Omega)} \|\nabla u^{\rm min}_h\|^2
\\
&\le&\displaystyle
C \frac{k}{h^2}\mathcal{F}(u_h^0,v_h^0)\|u^{\rm min}_h\|^2_h+\frac{k}{2} \|\nabla u^{\rm min}_h\|^2
\\
&\le&\displaystyle
\frac{1}{2}  \|u^{\rm min}_h\|^2_h +\frac{k}{2} \|\nabla u^{\rm min}_h\|^2;
\end{array}
$$
thereby
$$
\|u_h^{\rm min}\|_h^2\le0,
$$
which, in turn,  implies that $u^{\rm min}_h\equiv0$ and hence  $u^{n+1}_h\ge 0$. It remains to prove that indeed $u^{n+1}_h>0$. We proceed by contradiction. Let $\widetilde\a\in \mathcal{N}_h$ be such $u^{n+1}_h(\widetilde\a)=0$. Substitute  $\bar x_h=\varphi_{\widetilde\a}$ into \eqref{eq:u_h} to arrive at
$$
\begin{array}{rcl}
\displaystyle
0<u^{n}_h(\widetilde\a)\int_\Omega \varphi_{\widetilde\a}&=&(\nabla u_h^{n+1},\nabla \varphi_{\widetilde\a} )-(\nabla v_h^n, \overline{u}_h^{n+1}\nabla \varphi_{\widetilde\a})
\\
&=&\displaystyle
\sum_{\widetilde\a\not=\a\in\mathcal{N}_h} u^{n+1}_h(\a) \Big\{  (\nabla \varphi_\a, \nabla\varphi_{\widetilde\a})-(\overline{\varphi}_\a\nabla v^n_h, \nabla\varphi_{\widetilde \a})\Big\}\le0.
\end{array}
$$
In the last line we have utilized  \eqref{lm5.1-lab2} and the fact that $u^{n+1}_h\ge 0$. This gives a contradiction.

It is now a simple matter to show that \eqref{positivity-vh} holds. It completes the proof.
\end{proof}

We are now concerned with obtaining $L^1(\Omega)$ bounds for $(u^{n+1}_h, v^{n+1}_h)$. In particular, we will see that equation \eqref{eq:u_h} is mass-preserving.
\begin{lemma}[$L^1(\Omega)$-bounds] Under the conditions of Lemma \ref{lm:Lower_Bounds}, the discrete solution pair $(u^{n+1}_h, v_h^{n+1})\in X_h^2$ computed via \eqref{eq:u_h} and \eqref{eq:v_h} fulfills
\begin{equation}\label{Local-L1-Bound-uh}
\|u^{n+1}_h\|_{L^1(\Omega)}=\|u_h^0\|_{L^1(\Omega)}
\end{equation}
and
\begin{equation}\label{Local-L1-Bound-vh}
\|v^{n+1}_h\|_{L^1(\Omega)}\le \|v^0_h\|_{L^1(\Omega)}+\|u^0_h\|_{L^1(\Omega)}.
\end{equation}
\end{lemma}
\begin{proof} On choosing $x_h=1$ in \eqref{eq:u_h}, it follows immediately after a telescoping cancellation that
\begin{equation}\label{lm5.2-lab1}
\int_\Omega u^{n+1}_h(\x)\,\dx=\int_\Omega u^0_h(\x)\,\dx.
\end{equation}
Consequently, we get that \eqref{Local-L1-Bound-uh} holds from \eqref{positivity-uh} and \eqref{cond:u^0_h}. Now let $x_h=1$ in \eqref{eq:v_h} to get
$$
\int_\Omega v^{n+1}_h(\x)\,\dx+k \int_\Omega v^{n+1}_h(\x)\,\dx=\int_\Omega v^n_h(\x)\,\dx+k\int_\Omega u^{n+1}_h(\x)\,\dx.
$$
A simple calculation shows that
$$
\int_\Omega v^{n+1}_h(\x)\,\dx=\frac{1}{(1+k)^{n+1}} \int_\Omega v^0_h(\x)\,\dx+\left(\int_\Omega u^0_h(\x)\,\dx\right) \sum_{j=1}^{n+1}\frac{k}{(1+k)^j} ,
$$
where we have used \eqref{lm5.2-lab1}. Inequality \eqref{Local-L1-Bound-vh} is proved by applying \eqref{positivity-vh}.
\end{proof}
Once the positivity of $u^{n+1}_h$ has been proved, we are in a position to reformulate equation \eqref{eq:u_h} so as to be able to obtain a discrete energy law, which exactly mimics its counterpart at the continuous level.
\begin{lemma} Under the conditions of Lemma \ref{lm:Lower_Bounds}.  Equation \eqref{eq:u_h} can be equivalently written as
\begin{equation}\label{eq:u_h-equiv_II}
(\delta_t u^{n+1}_h, x_h)_h+(\nabla u_h^{n+1},\nabla x_h)=(\nabla v_h^n, \mathcal{A}_h(u_h^{n+1})\nabla x_h),
\end{equation}
where $\mathcal{A}_h(u^{n+1}_h)\in\mathds{R}^{2\times 2}$ is a piecewise constant, diagonal matrix defined as follows. Let $T\in\mathcal{T}_h$. Then there exist two pairs  $(\a_{\underline{u}^i}^T, \a_{\overline{u}^i}^T)\in T^2$, $i=1,2$, such that
\begin{equation}\label{def-A}
[\mathcal{A}_h(u^{n+1}_h)|_T]_{ii}=\left\{
\begin{array}{ccl}
\displaystyle
\frac{u^{n+1}_h(\a_{\overline{u}^i}^T)- u^{n+1}_h(\a_{\underline{u}^i}^T)}{f'(u^{n+1}_h(\a_{\overline{u}^i}^T))- f'(u^{n+1}_h(\a_{\underline{u}^i}^T))}&\mbox{ if }&u^{n+1}_h(\a_{\underline{u}^i}^T)- u^{n+1}_h(\a_{\overline{u}^i}^T)\not=0 ,
\\
\displaystyle
u^{n+1}_h(\boldsymbol{b}_T)&\mbox{ if }&u^{n+1}_h(\a_{\underline{u}^i}^T)- u^{n+1}_h(\a_{\overline{u}^i}^T)=0,
\end{array}
\right.
\end{equation}
where $f(s)=s \log s -s$.
\end{lemma}
\begin{proof} We must identify $\mathcal{A}_h(u^{n+1}_h)\equiv I_{2} \overline{u}^{n+1}_h$, where $I_2$ is the $2\times2$ identity matrix. To do so, first observe that $f'(s)=\log s$. Then it is easy to see that, for each $\varepsilon >0$ and $c\in (0, \infty)$, there exist two points $0<\underline{u}< \overline{u}$ such that $|\underline{u}- \overline{u}|<\varepsilon$, with $\underline{u}\le c\le \overline{u}$, so that
$$
\frac{f'(\overline{u})-f'(\underline{u})}{\overline{u}- \underline{u}}=\frac{1}{c}.
$$
Let $T\in\mathcal{T}_h$ and choose $c=u^{n+1}_h(\boldsymbol{b}_T)$. We are allowed to choose $\varepsilon$ small enough such that $$\underline{u}\ge \min_{t\in(-1, 1)} u^{n+1}_h( \boldsymbol{b}_T+  r_{\boldsymbol{b}_T}\boldsymbol{e}_i t)\not = \max_{t\in(-1, 1)} u^{n+1}_h( \boldsymbol{b}_T+  r_{\boldsymbol{b}_T}\boldsymbol{e}_i t)\ge\overline{u},$$ where $r_{\boldsymbol{b}_T}= dist(\boldsymbol{b}_T, \partial T)$ and $\boldsymbol{e}_i$ is the i$th$ vector of the canonical basis of $\R^2$. Therefore, there exists a pair $(\a_{\underline{u}^i}^T, \a_{\overline{u}^i}^T)$ such that
$u^{n+1}_h(\a_{\underline{u}^i}^T)=\underline{u}$ and $u^{n+1}_h(\a_{\overline{u}^i}^T)=\overline{u}$ and hence one defines $$[\mathcal{A}_h(u^{n+1}_h)|_T]_{ii}=\frac{u^{n+1}_h(\a_{\overline{u}^i}^T)- u^{n+1}_h(\a_{\underline{u}^i}^T)}{f'(u^{n+1}_h(\a_{\overline{u}^i}^T))- f'(u^{n+1}_h(\a_{\underline{u}^i}^T))}.$$
In the case that $\min_{t\in(-1, 1)} u^{n+1}_h( \boldsymbol{b}_T+  r_{\boldsymbol{b}_T}\boldsymbol{e}_i t) = \max_{t\in(-1, 1)} u^{n+1}_h( \boldsymbol{b}_T+  r_{\boldsymbol{i}_T}\boldsymbol{e}_i t)$, one defines
$$
[\mathcal{A}_h(u^{n+1}_h)|_T]_{ii}=u^{n+1}_h(\boldsymbol{b}_T). 
$$
This completes the proof.
\end{proof}

It is now shown that system \eqref{eq:u_h}-\eqref{eq:v_h} enjoys a discrete energy law locally in time.

\begin{lemma}[A discrete energy law] Under the conditions of Lemma \ref{lm:Lower_Bounds}, the discrete solution $(u^{n+1}_h, v^{n+1}_h)\in X_h^2$ computed via \eqref{eq:u_h} and \eqref{eq:v_h} satisfies
\begin{equation}\label{Local-Energy-Law}
\begin{array}{rcl}
\mathcal{E}_0(u^{n+1}_h,v^{n+1}_h)-\mathcal{E}_0(u^{n}_h,v^{n}_h)+k\|\delta_t v^{n+1}_h\|^2_h&&
\\
+k\|\mathcal{A}^{-\frac{1}{2}}_h(u^{n+1}_h)\nabla u^{n+1}_h-\mathcal{A}^{\frac{1}{2}}_h(u^{n+1}_h)\nabla v^n_h\|^2&\le&0
\end{array}
\end{equation}
where $\mathcal{E}_0(\cdot, \cdot)$ is defined in \eqref{def:energy}.
\end{lemma}
\begin{proof} First of all, recall that $f(s)=s \log s -s$; therefore, we are allowed to compute $f'(u^{n+1}_h)$ due to \eqref{positivity-uh}.
Select $x_h=\mathcal{I}_h f'(u^{n+1}_h)-v^n_h$ in \eqref{eq:u_h-equiv_II} and $x_h=\delta_t v^{n+1}_h$ in \eqref{eq:v_h} to get
\begin{equation}\label{lm5.3-lab1}
\begin{array}{r}
(\delta_t u^{n+1}_h, f'(u^{n+1}_h)-v^n_h)_h+(\nabla u^{n+1}_h, \nabla(\mathcal{I}_h f'(u^{n+1}_h)-v^n_h))
\\
-(\nabla v^{n}_h, \mathcal{A}_h(u^{n+1}_h)\nabla (\mathcal{I}_h f'(u^{n+1}_h)-v_h^n))=0
\end{array}
\end{equation}
and
\begin{equation}\label{lm5.3-lab2}
\begin{array}{l}
\displaystyle
\frac{1}{2 k}(\|v^{n+1}_h\|^2_h+\|\nabla v^{n+1}_h\|^2)-\frac{1}{2 k}(\|v^{n}_h\|^2_h+\|\nabla v^{n}_h\|^2)
\\
\displaystyle
+\frac{1}{2k}\|v^{n+1}_h-v^n_h\|^2_h+\frac{1}{2k}\|\nabla (v^{n+1}_h-v^n_h)\|^2
+\|\delta_t v^{n+1}_h\|^2_h- (u_h^{n+1}, \delta_t v^{n+1}_h)_h=0.
\end{array}
\end{equation}

We next pair some terms from \eqref{lm5.3-lab1} and \eqref{lm5.3-lab2} in order to handle them together. It is not hard to see that
\begin{equation}\label{lm5.3-lab3}
-(\delta_t u^{n+1}_h, v^n_h)_h-(u_h^{n+1}, \delta_t v^{n+1}_h)_h=-\frac{1}{k}(u^{n+1}_h, v^{n+1}_h)_h-\frac{1}{k}(u^n_h, v^n_h)_h.
\end{equation}

In view of \eqref{def-A}, there holds
$$
\partial_{x_i} \mathcal{I}_h(f'(u^{n+1}_h))=\frac{ f'(u^{n+1}_h(\a_{\overline{u}^i}^T))- f'(u^{n+1}_h(\a_{\underline{u}^i}^T))}{[\a_{\overline{u}^i}^T-\a_{\underline{u}^i}^T]_{i}},
$$
since $\mathcal{I}_h$ is piecewise linear. As a result,  one deduces from \eqref{def-A} that
$$
\nabla \mathcal{I}_h f'(u^{n+1}_h)= \mathcal{A}^{-1}_h(u^{n+1}_h)\nabla u^{n+1}_h.
$$
Therefore,
$$
(\nabla u^{n+1}_h, \nabla(\mathcal{I}_h f'(u^{n+1}_h)-v^n_h))=(\nabla u^{n+1}_h, \mathcal{A}^{-1}_h(u^{n+1}_h) \nabla u^{n+1})-(\nabla u^{n+1}_h, \nabla v^{n}_h)
$$
and
$$
(\nabla v^{n}_h, \mathcal{A}_h(u^{n+1}_h)\nabla (\mathcal{I}_h f'(u^{n+1}_h)-v_h^n))=(\nabla v^n_h, \nabla u^{n+1}_h)-(\nabla v^n_h, \mathcal{A}_h(u^{n+1}_h) \nabla v^n_h),
$$
which imply that
\begin{equation}\label{lm5.3-lab4}
\begin{array}{r}
(\nabla u^{n+1}_h, \nabla(\mathcal{I}_h f'(u^{n+1}_h)-v^n_h))-(\nabla v^{n}_h, \mathcal{A}_h(u^{n+1}_h)\nabla (\mathcal{I}_h f'(u^{n+1}_h)-v_h^n))
\\
= \|\mathcal{A}^{-\frac{1}{2}}_h(u^{n+1}_h)\nabla u^{n+1}_h-\mathcal{A}^{-\frac{1}{2}}_h(u^{n+1}_h)\nabla v^n_h\|^2.
\end{array}
\end{equation}

A Taylor polynomial of $f$ round $u^{n+1}_h$ evaluated at $u^{n}_h$ yields
$$
f(u^n_h)=f(u^{n+1}_h)-f'(u^{n+1}_h)(u^{n+1}_h-u^n_h)+\frac{f''(u^{n+\theta}_h)}{2}(u^{n+1}_h-u^n_h)^2,
$$
where $\theta\in (0,1)$ such that $u^{n+\theta}_h=\theta u^{n+1}_h+(1-\theta) u^n_h$. Hence,
$$
(\partial_t u^{n+1}_h, f'(u^{n+1}_h))_h=\frac{1}{k}(f(u^{n+1}_h),1)_h- \frac{1}{k}(f(u^{n+1}_h),1)_h+ \frac{k}{2} (f''(u^{n+\theta}_h),(\delta_t u^{n+1}_h)^2)_h.
$$
In fact, one can write the above expression as
\begin{equation}\label{lm5.3-lab5}
(\partial_t u^{n+1}_h, f'(u^{n+1}_h))_h=(u^{n+1}_h, \log u^{n+1}_h)_h-(u^n_h, \log u^n_h)_h+ k (f''(u^{n+\theta}_h),(\delta_t u^{n+1}_h)^2)_h
\end{equation}
owing to
$$
\int_{\Omega} u^{n+1}_h(\x)\,\dx=\int_{\Omega} u^{n}_h(\x)\,\dx.
$$

On adding \eqref{lm5.3-lab1} and \eqref{lm5.3-lab2}, we verify \eqref{Local-Energy-Law} from \eqref{lm5.3-lab3}, \eqref{lm5.3-lab4}, and \eqref{lm5.3-lab5}.

\end{proof}

\subsection{\emph{A priori} bounds}
Now that we have accomplished the discrete energy law \eqref{Local-Energy-Law} for system \eqref{eq:u_h}--\eqref{eq:v_h}, our goal is to derive \emph{a priori} energy bounds. It will be no means obvious since $\mathcal{E}_0(u^{n+1}_h,v^{n+1}_h)$ does not provide directly any control over $u^{n+1}_h$ and $v^{n+1}_h$.   The key ingredient will be the discrete Moser--Trudinger inequality \eqref{Moser-Trudinger-Ih}.
\begin{lemma}[Control of $(u^{n+1}_h, v^{n+1}_h)_h$]\label{lm:control_of_(u,v)_h} Assume that the conditions of Lemma \ref{lm:Lower_Bounds} are satisfied. Let $u_0\in H^1(\Omega)$ such that $\|u_0\|_{L^1(\Omega)}\in (0,4 \theta_\Omega)$. Then there exist $\delta,\varepsilon \in (0,1)$ such that
\begin{equation}\label{control_of_(u,v)_h}
 (u^{n+1}_h, v^{n+1}_h)_h\le \frac{1}{\delta}\mathcal{E}_0(u^{n+1}_h, v^{n+1}_h)+ \mathcal{R}^{\varepsilon, \delta}_0(u^0_h,v_h^0),
\end{equation}
where $\mathcal{E}_0(\cdot, \cdot)$ and  $\mathcal{R}_0^{\varepsilon, \delta}(\cdot,\cdot)$ are given in \eqref{def:energy} and  \eqref{def:R_0}, respectively.
\end{lemma}
\begin{proof} Let $\delta,\varepsilon\in(0,1)$ such that
\begin{equation}\label{Asump:delta-epsilon}
\frac{(1+\delta)^2[8\theta_\Omega C_{\rm MT}\varepsilon+ 1]\|u_0\|_{L^1(\Omega)}}{8\theta_\Omega}\le \frac{1}{2}.
\end{equation}
Using Jensen's inequality and invoking \eqref{Local-L1-Bound-uh}, one finds
$$
\begin{array}{rcl}
\displaystyle
- \log \int_\Omega \frac{\mathcal{I}_h (e^{(1+\delta)v^{n+1}_h(\x)})}{\|u^0_h\|_{L^1(\Omega)}}\,\dx&=&\displaystyle
-\log \sum_{\a\in\mathcal{N}_h} \frac{e^{(1+\delta)v^{n+1}_h(\a)}}{u^{n+1}_h(\a)} \frac{u^{n+1}_h(\a)}{\|u_h^0\|_{L^1(\Omega)}} \int_\Omega \varphi_\a(\x)\, \dx
\\
&\le & \displaystyle
\sum_{\a\in\mathcal{N}_h} -\log \left(\frac{e^{(1+\delta) v^{n+1}_h(\a)}}{u^{n+1}_h(\a)}\right) \frac{u^{n+1}_h(\a)}{\|u_h^0\|_{L^1(\Omega)}} \int_\Omega \varphi_a(\x)\,\dx
\\
&=&\displaystyle
-\frac{1+\delta}{\|u_h^0\|_{L^1(\Omega)}}(u^{n+1}_h,v^{n+1}_h)_h
\\
&&\displaystyle
+\frac{1}{\|u_h^0\|_{L^1(\Omega)}}(\log u_h^{n+1}, u^{n+1}_h)_h
\end{array}
$$
and hence
$$
\begin{array}{rcl}
-(u^{n+1}_h, v^{n+1}_h)_h+(\log u^{n+1}_h, u^{n+1}_h)_h&\ge& \delta (u^{n+1}_h, v^{n+1}_h)_h+\|u_h^0\|_{L^1(\Omega)} \log \|u_h^0\|_{L^1(\Omega)}
\\
&&\displaystyle
- \|u_h^0\|_{L^1(\Omega)} \log \int_\Omega \mathcal{I}_h (e^{(1+\delta)v^{n+1}_h(\x)})\,\dx.
\end{array}
$$
In virtue of \eqref{Moser-Trudinger-Ih}, we can bound
$$
\begin{array}{rcl}
\displaystyle
\log \int_\Omega \mathcal{I}_h (e^{(1+\delta)v^{n+1}_h(\x)})\,\dx&\le&  \log \Big(C_\Omega (1+C_{\rm MT} (1+\delta)^2 \|\nabla v^{n+1}_h\|^2)\Big)
\\
&&\displaystyle
+\frac{(1+\delta)^2}{8\theta_\Omega} \|\nabla v^{n+1}_h\|^2+\frac{1+\delta}{|\Omega|} \|v_h^{n+1}\|_{L^1(\Omega)}
\\
&=& \displaystyle
\log( \frac{C_\Omega}{\varepsilon}) +\log \Big(\varepsilon (1+C_{\rm MT} (1+\delta)^2 \|\nabla v^{n+1}_h\|^2)\Big)
\\
&&\displaystyle
+\frac{(1+\delta)^2}{8\theta_\Omega} \|\nabla v^{n+1}_h\|^2+\frac{1+\delta}{|\Omega|} \|v_h^{n+1}\|_{L^1(\Omega)}
\\
&\le& \displaystyle
\frac{C_\Omega}{\varepsilon} + \varepsilon+C_{\rm MT} \varepsilon (1+\delta)^2 \|\nabla u^{n+1}_h\|^2
\\
&&\displaystyle
+\frac{(1+\delta)^2}{8\theta_\Omega} \|\nabla v^{n+1}_h\|^2+\frac{1+\delta}{|\Omega|} \|v_h^{n+1}\|_{L^1(\Omega)}.
\end{array}
$$
Therefore,
$$
\begin{array}{rcl}
-(u^{n+1}_h, v^{n+1}_h)_h+(\log u^{n+1}_h, u^{n+1}_h)_h&\ge& \delta (u^{n+1}_h, v^{n+1}_h)_h+\|u_h^0\|_{L^1(\Omega)} \log \|u_h^0\|_{L^1(\Omega)}
\\
&& \displaystyle
-\|u_h^0\|_{L^1(\Omega)} (\frac{C_\Omega}{\varepsilon}+\varepsilon)
\\
&&\displaystyle
-C_{\rm MT} \varepsilon (1+\delta)^2 \|u_h^0\|_{L^1(\Omega)} \|\nabla v^{n+1}_h\|^2
\\
&&\displaystyle
- \frac{(1+\delta)^2 \|u_h^0\|_{L^1(\Omega)} }{8\theta_\Omega} \|\nabla v^{n+1}_h\|^2
\\
&&\displaystyle
-\frac{(1+\delta)\|u_h^0\|_{L^1(\Omega)}}{|\Omega|} \|v_h^{n+1}\|_{L^1(\Omega)}.
\end{array}
$$
On recalling \eqref{def:energy} and on noting  \eqref{Asump:delta-epsilon}, it follows from $x \log x>-\frac{1}{e}$ for $x>0$ that
$$
\begin{array}{rcl}
\mathcal{E}_0(u^{n+1}_h, v^{n+1}_h)&\ge&\displaystyle \frac{1}{2}\|v^{n+1}_h\|_h^2+
\frac{1}{2} \|\nabla v^{n+1}_h\|^2 +\delta (u^{n+1}_h, v^{n+1}_h)_h
\\
&&\displaystyle
+\|u_h^0\|_{L^1(\Omega)} \Big (\log \|u_h^0\|_{L^1(\Omega)} - \frac{C_\Omega}{\varepsilon} - \varepsilon\Big)
\\
&&\displaystyle
- \frac{(1+\delta)^2[8\theta_\Omega C_{\rm MT}\varepsilon+1]\|u_0\|_{L^1(\Omega)}]}{8\theta_\Omega} \|\nabla v^{n+1}_h\|^2
\\
&&\displaystyle
-\frac{(1+\delta)\|u_h^0\|_{L^1(\Omega)}}{|\Omega|} \|v_h^{n+1}\|_{L^1(\Omega)}
\\
&\ge&\displaystyle
\frac{1}{2} \| v^{n+1}_h\|^2_h +\delta (u^{n+1}_h, v^{n+1}_h)_h-\frac{1}{e}
\\
&&\displaystyle
-\|u_h^0\|_{L^1(\Omega)}\left( \frac{C_\Omega}{\varepsilon} +\varepsilon+\frac{(1+\delta)}{|\Omega|} \|v_h^{n+1}\|_{L^1(\Omega)}\right).
\end{array}
$$
Finally, we have, from \eqref{Local-L1-Bound-vh}, that
$$
\begin{array}{rcl}
\delta (u^{n+1}_h, v^{n+1}_h)_h&\le&\displaystyle
\mathcal{E}_0(u^{n+1}_h, v^{n+1}_h)+\frac{1}{e}
\\
&&\displaystyle
+\|u_h^0\|_{L^1(\Omega)}\Big(\frac{C_\Omega}{\varepsilon}+\varepsilon+\frac{(1+\delta)}{|\Omega|} (\|v_h^0\|_{L^1(\Omega)}+\|u^0_h\|_{L^1(\Omega)})\Big);
\end{array}
$$
thus, proving the result.
\end{proof}
The following is an immediate consequence of Lemma \ref{lm:control_of_(u,v)_h}.
\begin{corollary}[Control of $\|u_h^{n+1}\log u_h^{n+1}\|_{L^1(\Omega)}$] Under the conditions of Lemma \ref{lm:Lower_Bounds}, there holds
\begin{equation}\label{control_of_(logu,u)_h}
\|u_h^{n+1}\log u_h^{n+1}\|_{L^1(\Omega)}\le(1+\frac{1}{\delta})\mathcal{E}_0(u^{n+1}_h, v^{n+1}_h)+\mathcal{R}^{\varepsilon,\delta}_0(u^0_h,v^0_h)+2\frac{|\Omega|}{e}.
\end{equation}
\begin{proof} Write
\begin{equation}\label{co5.5-lab1}
\begin{array}{rcl}
\displaystyle
\|u^{n+1}_h \log u^{n+1}\|_{L^1(\Omega)}&=&\displaystyle
-\int_{\Omega} u^{n+1}_h(\x) \log_{-} u^{n+1}(\x)\,\dx
\\[1.5ex]
&&+\displaystyle
\int_{\Omega} u^{n+1}_h(\x) \log_{+} u^{n+1}(\x)\,\dx.
\end{array}
\end{equation}
Clearly, from $-\frac{1}{e}\le x \log x\le 0$ for $x\in[0,1]$, one gets
\begin{equation}\label{co5.5-lab2}
-\int_{\Omega} u^{n+1}_h(\x) \log_{-} u^{n+1}(\x)\,\dx\le \frac{|\Omega|}{e}.
\end{equation}
By the definition of $\mathcal{E}_0(u^{n+1}_h, v^{n+1}_h)$, one can easily deduce from \eqref{control_of_(u,v)_h} that
\begin{equation}\label{co5.5-lab3}
\begin{array}{rcl}
\displaystyle
\int_{\Omega} u^{n+1}_h(\x) \log_{+} u^{n+1}(\x)\,\dx&\le&\displaystyle
\mathcal{E}_0(u^{n+1}_h, v^{n+1}_h)
-\int_{\Omega} u^{n+1}_h(\x) \log_{-} u^{n+1}(\x)\,\dx
\\
&&+(u^{n+1}_h,v^{n+1}_h)_h
\\
&\le&\displaystyle
(1+\frac{1}{\delta})\mathcal{E}_0(u^{n+1}_h, v^{n+1}_h)+\mathcal{R}^{\varepsilon,\delta}_0(u^0_h,v^0_h)+\frac{|\Omega|}{e}.
\end{array}
\end{equation}
Thus, inserting \eqref{co5.5-lab2} and \eqref{co5.5-lab3} into \eqref{co5.5-lab1} yields \eqref{control_of_(logu,u)_h}.
\end{proof}
\end{corollary}

At this point a local-in-time,  \emph{a priori} bound for $u^{n+1}_h$ and $v^{n+1}_h$ on which an induction procedure will be applied is derived.
\begin{lemma}[\emph{A priori} bounds]\label{lm:Local_A_Priori_Bounds} Suppose that the conditions of Lemma \ref{lm:Lower_Bounds} are fulfilled. Let  $(u^{n+1}_h, v_h^{n+1})\in X_h^2$ be the discrete solution computed via \eqref{eq:u_h} and \eqref{eq:v_h}. Then there holds
\begin{equation}\label{Local_A_Priori_Bounds}
\begin{array}{ll}
\mathcal{E}_1(u^{n+1}_h, v^{n+1}_h)-\mathcal{E}_1(u^n_h, v^n_h)+2k \|\widetilde\Delta_h v^{n+1}_h\|^2
\\
\displaystyle
+k (\mathcal{R}^{\gamma,h}_1(u^{n+1}_h,v^{n+1}_h) \|\nabla u^{n+1}_h\|^2+\frac{1}{2}\|\nabla\widetilde\Delta_h v^{n+1}_h\|^2)
\\
\le C\,k \Big(\|\delta_t v_h^{n+1}\|_h+\|\delta_t v_h^{n+1}\|_h^2\Big)\|u^{n+1}_h\|^2_h
\\
\displaystyle
+C\,k \|u^{n+1}_h\log u^{n+1}_h\|_{L^1(\Omega)}+C k \|u^0_h\|_{L^1(\Omega)},
\end{array}
\end{equation}
where $\mathcal{E}_1(\cdot, \cdot)$ is defined in \eqref{def:energy_II} and
$$
\mathcal{R}^{\gamma,h}_1(u^{n+1}_h,v^{n+1}_h):=\frac{4}{3}-\gamma-\gamma^3 \|u^{n+1}_h\log u^{n+1}_h\|_{L^1(\Omega)}- C h^{1-\frac{2}{p}} \mathcal{F}(u^0_h,v^0_h).
$$
\end{lemma}
\begin{proof} Set $ x_h=u^{n+1}_h$ in \eqref{eq:u_h} to obtain
\begin{equation}\label{lm5.6-lab1}
\|u^{n+1}_h\|^2_h-\|u^n_h\|_h^2+\|u^{n+1}_h-u^n_h\|_h^2+2 k\|\nabla u_h^{n+1}\|^2=2\, k(\nabla v_h^n, \overline{u}_h^{n+1}\nabla u_h^{n+1}).
\end{equation}
Consider $T\in\mathcal{T}_h$ and let $\boldsymbol{b}_T$ be its barycenter to write
$$
u_h^{n+1}(\boldsymbol{b}_T)=u_h^{n+1}(\a_{\overline{u}^i}^T)+\nabla u_h^{n+1}|_T\cdot (\boldsymbol{b}_T-\a_{\overline{u}^i}^T)
$$
and
$$
u_h^{n+1}(\boldsymbol{b}_T)=u_h^{n+1}(\a_{\underline{u}^i}^T)+\nabla u_h^{n+1}|_T\cdot (\boldsymbol{b}^T-\a_{\underline{u}^i}^T);
$$
thereby,
\begin{equation}\label{lm5.6-lab2}
u_h^{n+1}(\boldsymbol{b}_T)=\frac{1}{2} (u_h^{n+1}(\a_{\overline{u}^i}^T)+u_h^{n+1}(\a_{\underline{u}^i}^T))+\frac{1}{2} \nabla u_h^{n+1}|_T\cdot ( (\boldsymbol{b}^T-\a_{\overline{u}^i}^T) + (\boldsymbol{b}^T-\a_{\underline{u}^i}^T)).
\end{equation}
To deal with the right-hand side of \eqref{lm5.6-lab1}, we proceed as follows. Let us write
\begin{equation}\label{lm5.6-lab3}
\begin{array}{rcl}
(\nabla v_h^n, \overline{u}_h^{n+1}\nabla u_h^{n+1})&=&\displaystyle\sum_{i=1}^2 \sum_{T\in\mathcal{T}_h} \int_T \partial_{\x_i} v_h^n u^{n+1}_h(\boldsymbol{b}_T) \partial_{\x_i} u^{n+1}_h\, \dx.
\end{array}
\end{equation}
Thus, on substituting \eqref{lm5.6-lab2}  into \eqref{lm5.6-lab3}, we arrive at
\begin{align*}
\int_T \partial_{\x_i} v_h^n u^{n+1}_h(\boldsymbol{b}_T) \partial_{\x_i} u^{n+1}_h\, \dx
&=\displaystyle \int_T \partial_{\x_i} v_h^n u^{n+1}_h(\boldsymbol{b}_T) \frac{u_h^{n+1}(\a_{\overline{u}^i}^T)-u_h^{n+1}(\a_{\underline{u}^i}^T)}{(\a_{\overline{u}^i}^T-\a_{\underline{u}^i}^T)_{i}}\, \dx
\\
&= \frac{1}{2} \int_T \partial_{\x_i} v_h^n  (u_h^{n+1}(\a_{\overline{u}^i}^T)+u_h^{n+1}(\a_{\underline{u}^i}^T)) \frac{u_h^{n+1}(\a_{\overline{u}^i}^T)-u_h^{n+1}(\a_{\underline{u}^i}^T)}{(\a_{\overline{u}^i}^T-\a_{\underline{u}^i}^T)_{i}}\, \dx
\\
&+\int_T \partial_{\x_i} v_h^n \nabla u_h^{n+1}\cdot ( (\boldsymbol{b}_T-\a_{\overline{u}^i}^T) + (\boldsymbol{b}_T-\a_{\underline{u}^i}^T)) \partial_{\x_i} u^{n+1}_h \, \dx
\\
&=  \frac{1}{2} \int_T \partial_{\x_i} v_h^n \partial_{\x_i}\mathcal{I}_h (u_h^{n+1})^2\, \dx
\\
&+ \int_T \partial_{\x_i} v_h^n \nabla u_h^{n+1}\cdot ( (\boldsymbol{b}_T-\a_{\overline{u}^i}^T) + (\boldsymbol{b}_T-\a_{\underline{u}^i}^T)) \partial_{\x_i} u^{n+1}_h \, \dx,
\end{align*}
which combined with \eqref{lm5.6-lab3} shows that
\begin{equation}\label{lm5.6-lab4}
\begin{array}{rcl}
\displaystyle
(\nabla v_h^n, \overline{u}_h^{n+1}\nabla u_h^{n+1})
&=&\displaystyle\frac{1}{2} (\nabla v_h^{n+1},\nabla\mathcal{I}_h (u_h^{n+1})^2)
\\
&&\displaystyle+\frac{1}{2} (\nabla (v_h^n-v^{n+1}_h),\nabla\mathcal{I}_h (u_h^{n+1})^2)
\\
&&\displaystyle+ \sum_{i=1}^2 \sum_{T\in\mathcal{T}_h}\int_T \partial_{\x_i} v_h^n \nabla u_h^{n+1}\cdot  (\boldsymbol{b}_T-\a_{\overline{u}^i}^T) \partial_{\x_i} u^{n+1}_h \, \dx
\\
&&\displaystyle+\sum_{i=1}^2 \sum_{T\in\mathcal{T}_h}\int_T \partial_{\x_i} v_h^n \nabla u_h^{n+1}\cdot  (\boldsymbol{b}_T-\a_{\underline{u}^i}^T) \partial_{\x_i} u^{n+1}_h \, \dx
\\
&:=&I_1+I_2+I_3+I_4.
\end{array}
\end{equation}
It remains to bound each term of \eqref{lm5.6-lab4}. We first proceed with $I_1$. Choose $x_h=\frac{1}{2}\mathcal{I}_h (u^{n+1}_h)^2$ in \eqref{eq:v_h} to write
$$
\begin{array}{rcl}
I_1&=&\displaystyle
-\frac{1}{2}(\delta_t v^{n+1}_h, 	\mathcal{I}_h (u_h^{n+1})^2)_h
-\frac{1}{2}(v_h^{n+1}, \mathcal{I}_h (u_h^{n+1})^2)_h
+\frac{1}{2}(u^{n+1}_h,\mathcal{I}_h (u_h^{n+1})^2)_h
\\
&:=&J_1+J_2+J_3.
\end{array}
$$
An estimate for $J_1$ is easily computed from \eqref{Stab:I_h_Ln} for $n=2, 4$ and the Gagliardo-Nirenberg interpolation $\|x_h\|_{L^4(\Omega)}\le C \|x_h\|^{\frac{1}{2}} \|x_h\|_{H^1(\Omega)}^\frac{1}{2}$. It is given by
$$
\begin{array}{rcl}
J_1&\le& C \|\delta_t v_h^{n+1}\|_h \|u_h^{n+1}\|^2_{L^4(\Omega)}
\\
&\le&C \|\delta_t v_h^{n+1}\|_h ( \|\nabla u^{n+1}_h\| \|u^{n+1}_h\|_h+ \|u^{n+1}_h\|^2_h)
\\
&\le&\displaystyle
C (\|\delta_t v_h^{n+1}\|_h+\|\delta_t v_h^{n+1}\|_h^2)\|u^{n+1}_h\|^2_h+ \frac{\gamma}{4} \|\nabla u^{n+1}_h\|,^2
\end{array}
$$
where $\gamma>0$ is a constant to be adjusted later on. From \eqref{positivity-uh} and \eqref{positivity-vh}, we know that $J_2\le 0$. For $J_3$, we use the interpolation $\|x_h\|_{L^3(\Omega)}\le \widetilde\gamma \|\nabla x_h\|^{\frac{2}{3}} \|x_h \log |x_h|\|_{L^1(\Omega)}^{\frac{1}{3}}+C_{\widetilde\gamma}\|x_h \log |x_h|\|_{L^1(\Omega)}+ C_{\widetilde\gamma} \|x_h\|_{L^1(\Omega)}^{\frac{1}{3}}$ for $\widetilde\gamma>0$ (see \cite[Lemma 3.5]{Nagai_Senba_Yoshida_1997}) and \eqref{Stab:I_h_Ln} for $n=3$ to obtain
$$
\begin{array}{rcl}
J_3&\le& C \|u_h^{n+1}\|^3_{L^3(\Omega)}
\\
&\le&\displaystyle
\frac{\gamma^3}{2} \|u^{n+1}_h \log u^{n+1}_h\|_{L^1(\Omega)} \|\nabla u^{n+1}_h\|^2
\\
&&+C \|u^{n+1}_h\log u^{n+1}_h\|^3_{L^1(\Omega)}+C \|u^{n+1}_h\|_{L^1(\Omega)}.
\end{array}
$$
Inequality \eqref{inv_W1pToLp_global} for $p=2$ shows that
$$
\begin{array}{rcl}
\displaystyle
I_2&=&\displaystyle-\frac{k}{2}(\nabla \delta_t v^{n+1}_h, 	\nabla \mathcal{I}_h(u^{n+1}_h)^2)
\\
&\le&\displaystyle
C \frac{k}{h^2} \|\delta_t v^{n+1}_h\|_h \|u^{n+1}_h\|^2_{L^4(\Omega)}
\\
&\le&\displaystyle
C\frac{k}{h^2}(\|\delta_t v_h^{n+1}\|_h+\|\delta_t v_h^{n+1}\|_h^2)\|u^{n+1}_h\|^2_h+ \frac{\gamma}{4} \|\nabla u^{n+1}_h\|^2.
\end{array}
$$
We treat $I_3$ and $I_4$ together. Thus,
$$
\begin{array}{rcl}
I_3+I_4&\le& C h \|\nabla v^n_h\|_{L^\infty(\Omega)} \|\nabla u^{n+1}_h\|^2\le C  h^{1-\frac{2}{p}} \|\nabla v^{n}_h\|_{L^p(\Omega)} \|\nabla u^{n+1}_h\|^2
\\
&\le & C h^{1-\frac{2}{p}}  \|\widetilde\Delta_h v^n_h\|  \|\nabla u^{n+1}_h\|^2\le C h^{1-\frac{2}{p}} \mathcal{F}(u^0_h,v^0_h) \|\nabla u^{n+1}_h\|^2.
\end{array}
$$
In the above we used \eqref{inv_W1infToW1p_global}, \eqref{stab:Discrete-Laplacian-W1p-H2-xh}, and   \eqref{hyp:induction}. The estimates for the $I_i$'s applied to \eqref{lm5.6-lab1} lead to
\begin{equation}\label{lm5.6-lab5}
\begin{array}{rcl}
\|u^{n+1}_h\|^2_h&-&\|u^n_h\|_h^2+\|u^{n+1}_h-u^n_h\|_h^2+2 k\|\nabla u_h^{n+1}\|^2
\\
&\le&\displaystyle C(1+\frac{k}{h^2}) k (\|\delta_t v_h^{n+1}\|_h+\|\delta_t v_h^{n+1}\|_h^2)\|u^{n+1}_h\|^2_h
\\
&&+ k\Big(\gamma+\gamma^3 \|u^{n+1}_h \log u^{n+1}_h\|_{L^1(\Omega)}+ C h^{1-\frac{2}{p}} \mathcal{F}(u^0_h,v^0_h)\Big) \|\nabla u^{n+1}_h\|^2
\\
&&+Ck  \|u^{n+1}_h\log u^{n+1}_h\|^3_{L^1(\Omega)}+C k \|u^{n+1}_h\|_{L^1(\Omega)}.
\end{array}
\end{equation}
Choose $x_h=-\widetilde\Delta^2 v^{n+1}_h$ in \eqref{eq:v_h} to get
\begin{equation}\label{lm5.6-lab6}
\begin{array}{rcl}
\|\widetilde\Delta_h v^{n+1}_h\|_h^2-\|\widetilde \Delta_h v^n_h\|^2_h+\|\widetilde \Delta_h (v^{n+1}_h- v^n_h)\|_h^2&&
\\
+2k\|\nabla \widetilde\Delta_h v^{n+1}_h\|^2_h+2 k\| \widetilde\Delta_h v^{n+1}_hv^{n+1}_h\|_h^2&=&2 k(\nabla u^{n+1}_h, \nabla \widetilde \Delta_h v^{n+1}_h)
\\
&\le&\displaystyle
k(\frac{4}{3} \|\nabla u^{n+1}_h\|^2+\frac{3}{2}\|\nabla \widetilde \Delta_h u^{n+1}_h\|^2).
\end{array}
\end{equation}

The proof follows by use of \eqref{lm5.6-lab5} and \eqref{lm5.6-lab6}.
\end{proof}

\subsection{Induction argument}
The essential step to finishing up the proof of Theorem~\ref{Th:main} is an induction argument on $n$. We need to verify that the overall sequence $\{u_h^m\}_{m=0}^N$ provided by system \eqref{eq:u_h}-\eqref{eq:v_h} accomplishes the estimates from Theorem~\ref{Th:main}.

Observe first that $\mathcal{F}(u^0_h, v^0_h)$ is uniformly bounded  with regard to $h$, because of \eqref{cond:u^0_h} and \eqref{cond:v^0_h}, and hence we are allowed to choose $(h,k)$ satisfying \eqref{restriction-h-k} and \eqref{restriction-h}.

$\bullet$ \textbf{Case} ($m=1$). We want to prove Theorem \ref{Th:main} for $m=1$. Inequality \eqref{hyp:induction} holds trivially, since $\mathcal{F}(u^0_h, v^0_h)$ is bounded independently of $(h,k)$; thereby, from \eqref{positivity-uh} and \eqref{positivity-vh}, we obtain, for $n=0$, that, for all $\x\in\Omega$,
\begin{equation}\label{Th:lab1}
u^{1}_h(\x)>0
\end{equation}
and
\begin{equation}\label{Th:lab2}
u^{1}_h(\x)\ge0.
\end{equation}
Likewise, we have, by \eqref{Local-L1-Bound-uh} and \eqref{Local-L1-Bound-vh} for $n=0$,  that
$$
\|u^1_h\|_{L^1(\Omega)}=\|u^0_h\|_{L^1(\Omega)}
$$
and
$$
\|v^1_h\|_{L^1(\Omega)}\le \|v^0_h\|_{L^1(\Omega)}+ \|u^0_h\|_{L^1(\Omega)}.
$$
In view of \eqref{Th:lab1} and \eqref{Th:lab2},  inequality \eqref{Local-Energy-Law} for $n=0$ shows that
$$
\begin{array}{rcl}
\mathcal{E}_0(u^{1}_h,v^{1}_h)-\mathcal{E}_0(u^{0}_h,v^{0}_h)+k\|\delta_t v^{1}_h\|^2_h&&
\\
+k\|\mathcal{A}^{-\frac{1}{2}}_h(u^{1}_h)\nabla u^{1}_h-\mathcal{A}^{\frac{1}{2}}_h(u^{1}_h)\nabla v^0_h\|^2&\le&0
\end{array}
$$
which, in turn, gives
\begin{equation}\label{Th:lab3}
\mathcal{E}_0(u^1_h, v^1_h)\le \mathcal{E}_0(u^0_h, v^0_h).
\end{equation}
Applying \eqref{Th:lab3} to \eqref{control_of_(u,v)_h} and \eqref{control_of_(logu,u)_h}  for $n=0$ yields that
\begin{equation}\label{Th:lab4}
\begin{array}{rcl}
 (u^{1}_h, v^{1}_h)_h&\le&\displaystyle \frac{1}{\delta}\mathcal{E}_0(u^0_h, v^0_h)+ \mathcal{R}^{\delta,\varepsilon}_0(u^0_h,v_h^0)
\\
&:=&\mathcal{B}_0(u^0_h,v^0_h),
\end{array}
\end{equation}
and
\begin{equation}\label{Th:lab5}
\begin{array}{rcl}
\|u_h^1\log u_h^1\|_{L^1(\Omega)}&\le&\displaystyle (1+\frac{1}{\delta})\mathcal{E}_0(u^{0}_h, v^{0}_h)+\mathcal{R}^{\varepsilon,\delta}_0(u^0_h,v^0_h)+2\frac{|\Omega|}{e}
\\
&:=&\mathcal{B}_1(u^0_h,v^0_h).
\end{array}
\end{equation}
As a result of applying \eqref{Th:lab4} and \eqref{Th:lab5} to \eqref{Local-Energy-Law} for $n=0$, we find
$$
\begin{array}{rcl}
k\|\delta_t v^{1}_h\|^2_h&\le& \mathcal{E}_0(u^{0}_h,v^{0}_h)-\mathcal{E}_0(u^{1}_h,v^{1}_h)_h
\\
&\le&\displaystyle
\mathcal{E}_0(u^{0}_h,v^{0}_h)+\mathcal{B}_0(u^0_h,v_h^0)+\mathcal{B}_1(u^0_h,v_h^0)
\\
&:=&\mathcal{B}_2(u^0_h,v^0_h).
\end{array}
$$

Selecting $\gamma$ to be sufficiently small such that
\begin{equation}\label{Th:smallness_gamma}
\gamma+\gamma^3\mathcal{B}_2(u^0_h, v^0_h)\le \frac{5}{12}
\end{equation}
and recalling \eqref{restriction-h_II}, this implies, from \eqref{Th:lab3}, that
$$
\mathcal{R}^{\gamma,h}_1(u^1_h,v^1_h)\ge \frac{4}{3}-\gamma-\gamma^3\mathcal{B}_2(u^0_h, v^0_h)- C h^{1-\frac{1}{p}} \mathcal{E}_1(u^0_h,v^0_h)\ge\frac{1}{2};
$$
thus, one can find upon using  \eqref{Local_A_Priori_Bounds} for $n=0$ that
\begin{equation}
\begin{array}{ll}
\displaystyle
\mathcal{E}_1(u^1_h, v^1_h)-\mathcal{E}_1(u^0_h, v^0_h)+2k \|\widetilde\Delta_h v^1_h\|^2
+\frac{k}{2} ( \|\nabla u^1_h\|^2+\|\nabla\widetilde\Delta_h v^1_h\|^2)
\\
\le \displaystyle
C k (\|\delta_t v_h^{1}\|_h+\|\delta_t v_h^{1}\|_h^2)\|u^{1}_h\|^2
+ Ck \mathcal{B}_1^3(u^0_h, v^0_h)+C k \|u^0_h\|_{L^1(\Omega)}.
\end{array}
\end{equation}
Grönwall's inequality now provides the bound
$$
\mathcal{E}_1(u^{1}_h, v^{1}_h)+\frac{k}{2} ( \|\nabla u^{1}_h\|^2+\|\nabla\widetilde\Delta_h v^{1}_h\|^2)\le \mathcal{F}(u^0_h, v^0_h).
$$
Theorem \ref{Th:main} is therefore verified for $m=1$.

$\bullet$ \textbf{Case} $m=n+1$. Assume that the bounds in Theorem \ref{Th:main} are valid for all $m\in\{1, \cdots, n\}$. Consequently, it follows that
\begin{align}
\mathcal{E}_1(u^{n}_h, v^{n}_h)+&\frac{k}{2}\sum_{r=1}^{n} (\|\nabla u^{r}_h\|^2+\|\nabla\widetilde\Delta_h v^{r}_h\|^2)\nonumber
\\
&\le \mathcal{E}_1(u^0_h, v^0_h)+ C\,k\sum_{r=1}^{n} \Big(\|\delta_t v_h^{r}\|_h+\|\delta_t v_h^{r}\|_h^2\Big)\|u^{r}_h\|^2 \label{Th:induc_a_priori_bounds}
\\
\displaystyle
&+C\,k \sum_{r=1}^{n} \Big(\mathcal{B}_1^3(u^{0}_h, v^{0}_h)+ \|u^{0}_h\|_{L^1(\Omega)} \Big)
\nonumber
\end{align}
holds on the basis of
$$
\mathcal{R}^{\gamma,h}_1(u^m_h,v^m_h)\ge\frac{1}{2}\quad \mbox{ for all }\quad m\in\{0, \cdots, n\}.
$$
Then we want to prove Theorem \ref{Th:main} for $m=n+1$. Indeed, by the induction hypothesis \eqref{Global-Energy-Bound} for $m=n$, it is clear that \eqref{hyp:induction} holds; therefore, one has \eqref{positivity-uh} and \eqref{positivity-vh}.  That inequalities \eqref{Global-L1-Bound-uh} and \eqref{Global-L1-Bound-vh} are satisfied for $m=n+1$ is simply by noting \eqref{Local-L1-Bound-uh} and \eqref{Local-L1-Bound-vh}.
Combining \eqref{Local-Energy-Law} and the induction hypothesis \eqref{Global-Energy-Law} for $m=n$, we deduce \eqref{Global-Energy-Law} for $m=n+1$, which implies
\begin{equation}\label{Th:lab7}
\max_{m\in\{0,\cdots, n+1\}}\mathcal{E}_0(u^m_h,v^m_h)\le \mathcal{E}_0(u^0_h,u^0_h).
\end{equation}
As a result of this, we have, by \eqref{control_of_(u,v)_h} and \eqref{control_of_(logu,u)_h}, that
\begin{equation}\label{Th:lab8}
(u^{n+1}_h, v^{n+1}_h)_h\le \mathcal{B}_0(u^0_h,v^0_h)
\end{equation}
and
\begin{equation}\label{Th:lab9}
\|u_h^{n+1}\log u_h^{n+1}\|_{L^1(\Omega)}\le\mathcal{B}_1(u^0_h,v^0_h).
\end{equation}
Moreover, it follows from \eqref{Global-Energy-Law} for $m=n+1$ that
\begin{equation}\label{Th:lab10}
k\sum_{r=0}^{n+1}\|\delta_t v^{r}_h\|^2_h\le \mathcal{B}_2(u^0_h,v^0_h),
\end{equation}
in view of \eqref{Th:lab8} and \eqref{Th:lab9}.

Once again if $\gamma$ is chosen to be small enough such that \eqref{Th:smallness_gamma} holds and condition \eqref{restriction-h_II} is invoked, one finds
$$
\mathcal{R}^{\gamma,h}_1(u^{n+1}_h,v^{n+1}_h)\ge \frac{4}{3}-\gamma-\gamma^3\mathcal{B}_2(u^0_h, v^0_h)- C h^{1-\frac{1}{p}} \mathcal{E}_1(u^0_h,v^0_h)\ge\frac{1}{2}
$$
owing to \eqref{Th:lab7}. We thus infer from \eqref{Local_A_Priori_Bounds} combined with \eqref{Th:induc_a_priori_bounds} that 
\begin{align*}
\mathcal{E}_1(u^{n+1}_h, v^{n+1}_h)+\frac{k}{2}\sum_{r=1}^{n+1} (\|\nabla u^{r}_h\|^2+\|\nabla\widetilde\Delta_h v^{r}_h\|^2)
\\
\le \mathcal{E}_1(u^0_h, v^0_h)+ C\,k\sum_{r=1}^{n+1} (\|\delta_t v_h^{r}\|_h+\|\delta_t v_h^{r}\|_h^2)\|u^{r}_h\|^2
\\
\displaystyle
+C\,k \sum_{r=1}^{n+1} \Big(\mathcal{B}_1^3(u^{0}_h, v^{0}_h)+C k \|u^{r}_h\|_{L^1(\Omega)}\Big)
\end{align*}
and hence Grönwall's inequality provides \eqref{Global-Energy-Bound} for $m=n+1$ when used \eqref{Th:lab10}.

\section{Computational experiments}

\newcommand{\cteU}{\ensuremath{C_u}}
\newcommand{\cteV}{\ensuremath{C_v}}
\newcommand{\Nsquare}{N_{\mbox{square}}}

The computational experiments are meant to support and complement the theoretical results in the earlier sections in two different settings. On the one hand, we regard initial data $u_0$ under the condition $\int_\Omega u_0(\x)\,\dx\in (0, 4\pi)$, which give solutions remaining bounded over time. On the other hand, we use a particularly demanding test where a finite time blowup is expected. For this latter numerical test, it must be said that the blowup setting is out of reach from our analysis since \eqref{Global-Energy-Bound} is not satisfied for blowup solutions; therefore, lower bounds cannot be guaranteed. Nevertheless, the results are striking with regard to lower bounds since they fail very close to the expecting blowup time for not so small discrete parameters.

All the computations were performed with the help of the FreeFem++ framework~\cite{Hecht_2012}.

\subsection{Non-blowup setting}
\label{sec:non-blow-up}
As the domain we take the square $\overline\Omega = [-1/2,
1/2]^2$. The evolution starts from the bell-shaped initial data
\begin{equation}
  u_0 = \cteU e^{-\cteU(x^2+y^2)} \quad\mbox{and}\quad  v_0 = \cteV e^{-\cteV(x^2+(y-0.5)^2)},
  \label{eq:u0v0}
\end{equation}
which conditions fulfill a homogeneous Neumann boundary condition approximately\footnote{If $C_u$ and $C_v$ are quite small, the Neumann boundary condition on $(u_0, v_0)$ is not approximately null. For this reason, we only take \rf{$\cteU$} bigger or equal to 40.}. It should be noticed that $u_0$ is centered at the origin $(0,0)$, whereas $v_0$ is centered at the midpoint of  the top edge of the domain.

From now on, it is assumed that the constant $C_v$ and $C_u$ are the same. Then, for each $\cteU$, one can compute that $\int_\Omega u_0(\x)\,\dx\in (0,4\pi)$; therefore, problem \eqref{KS} with \eqref{IC}--\eqref{BC} has a unique, smooth solution. As a result of this experiment, we expect diffusion and chemotaxis transfer of cells (the $u$ component of the solution) from the center of the domain toward the top edge, where the highest concentration of chemical agent (the $v$ component of the solution) is found.

For the spacial discretization, we introduce the $\mathcal{P}_1$ finite element space $X_h$ associated with an \emph{acute} mesh $\mathcal{T}_h$ defined as follows. From an $\Nsquare\times\Nsquare$ uniform grid, obtained by dividing $\Omega$ into macroelements consisting of squares, we construct the mesh $\mathcal{T}_h$ by splitting each macroelement into $14$ acute triangles as indicated in Figure \ref{fig:acute_macrolement}.  This way, for $\Nsquare = 50$, we define a mesh consisting of 35,000 acute triangles and 17,701 vertices with mesh size $h\simeq 0.0101247$. Selecting $\cteU=70$, we compute $N=50$ time iterations using scheme~(\ref{eq:u_h})--(\ref{eq:v_h}) with time step $k=10^{-4}$. Snapshots of the simulations at times $t_n=0$, $2.5\cdot 10^{-3}$ and $5\cdot 10^{-3}$ are collected in Figure~\ref{fig:test1_cu70}. The same test is repeated for $\cteU\in\{40,50,60\}$, checking that positivity of the numerical solution is preserved over time iterations. In all these cases, the qualitative behavior expected in chemotaxis phenomena is obtained.
\begin{figure}
  \centering
  \includegraphics[width=0.33\linewidth, height=0.3\linewidth]{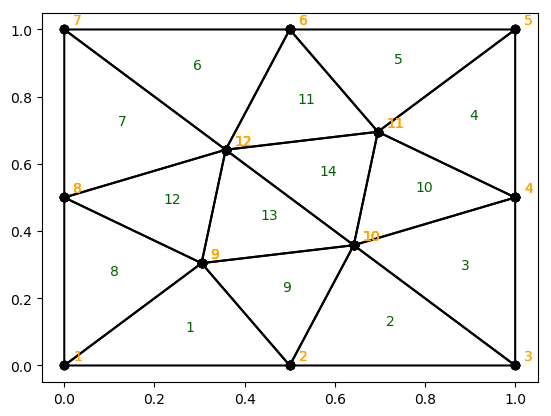}
  \caption{Reference macrolement, composed of 14 acute triangles}
  \label{fig:acute_macrolement}
\end{figure}
\begin{figure}
  \centering
  \begin{tabular}{c@{}c@{}c}
  \includegraphics[width=0.33\linewidth]{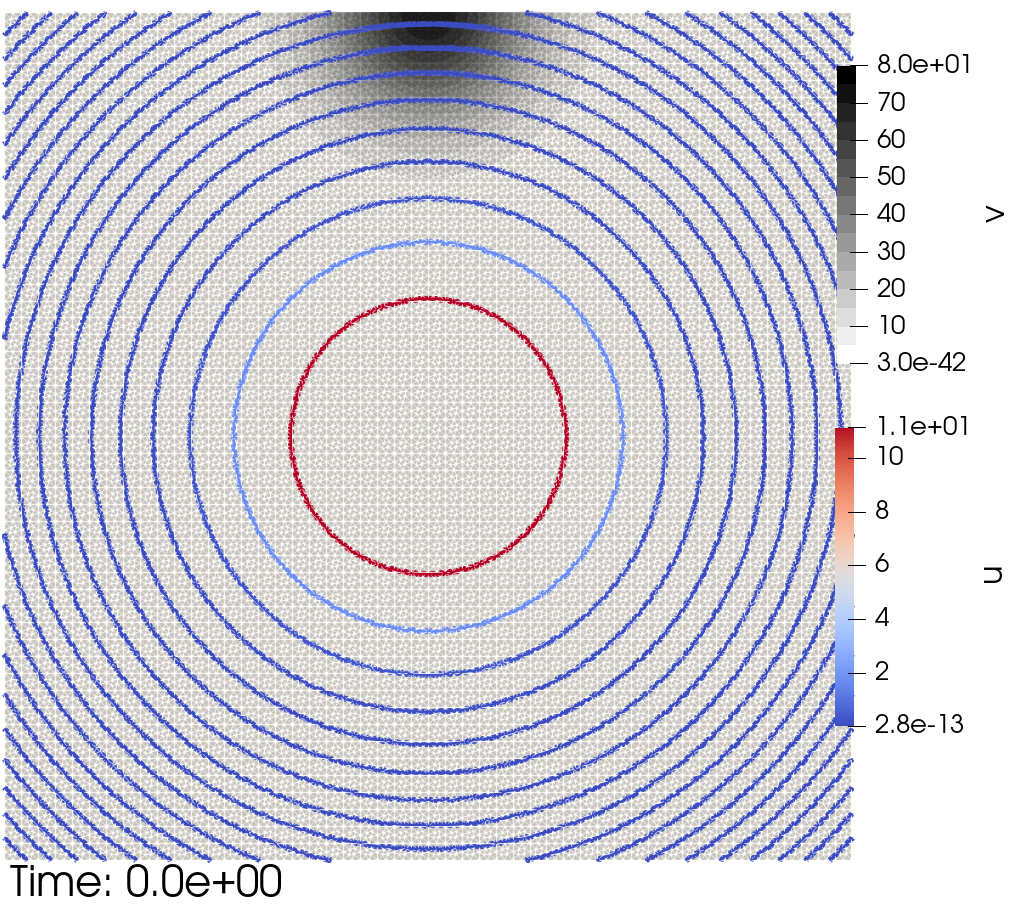}&
  \includegraphics[width=0.33\linewidth]{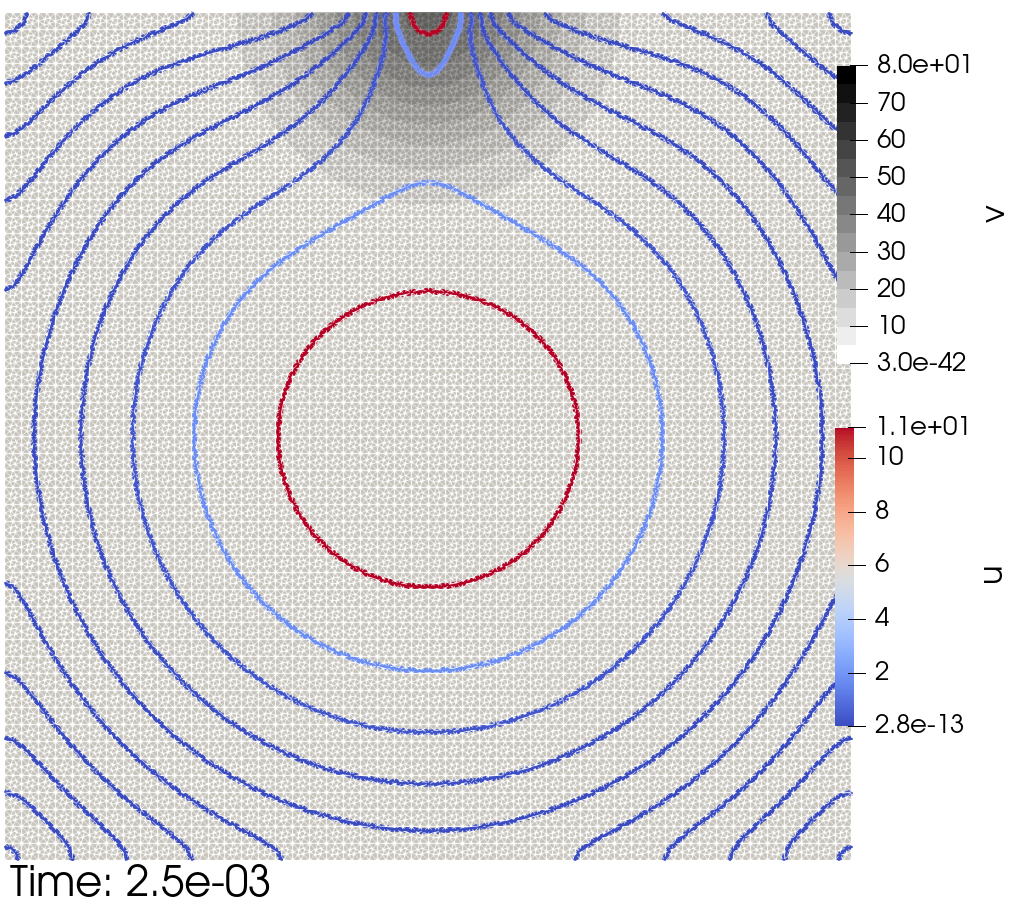}&
  \includegraphics[width=0.33\linewidth]{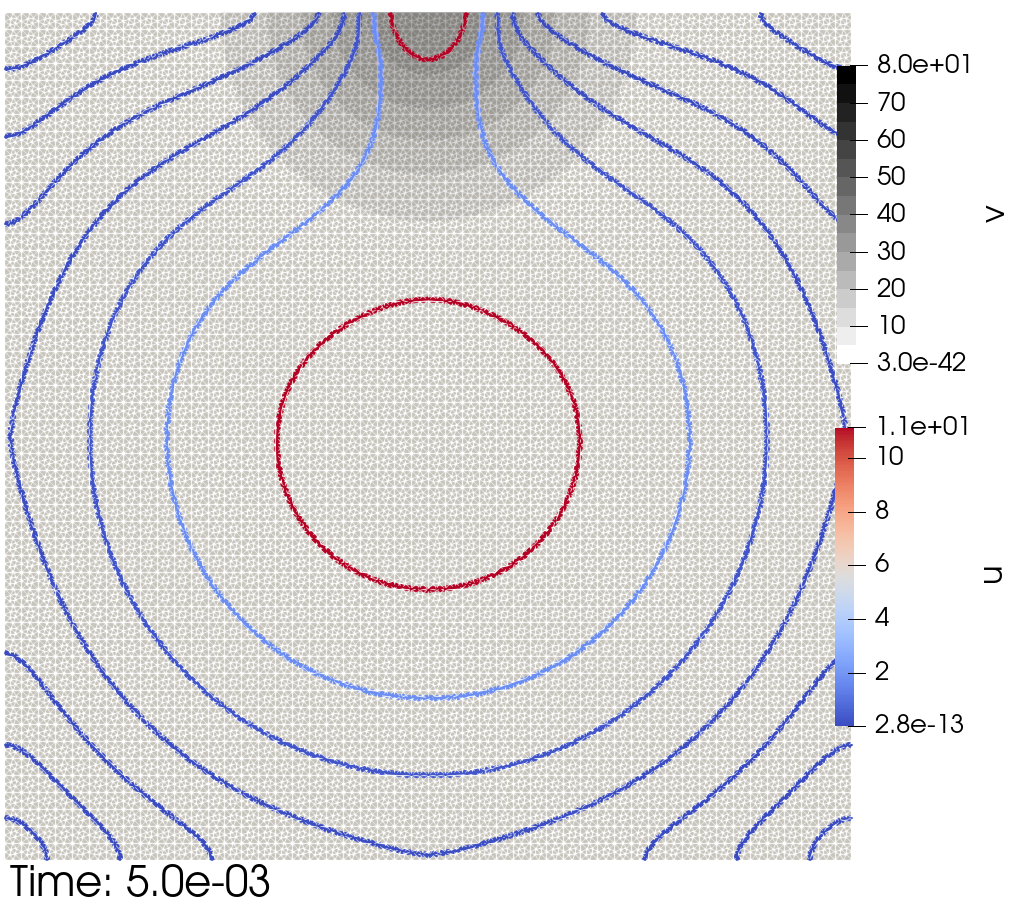}
  \end{tabular}
  \caption{Solution $u^n_h$ (colored isolines) and $v_h$ (gray scale
    background) at three time steps: $t_n=0$, $2.5\cdot 10^{-3}$ and
    $5\cdot 10^{-3}$. Diffusion
    and chemotaxis transfer of $u^n_h$ (cells) towards highest
    concentrations of $v^n_h$ can be seen along time. Acute mesh with
    $\Nsquare=50$, $k=10^{-4}$, initial data parameter: $\cteU=70$}.
  \label{fig:test1_cu70}
\end{figure}

Positivity of $u_h^{n+1}$ breaks if $\cteU$ grows beyond $\cteU\simeq 70$, but $v_h^{n+1}$ remains positive. Note that, as $\cteV=\cteU$ is increased, $\|\nabla v_h^0\|_{L^\infty(\Omega)}$ becomes larger and larger, with $v_0$ defined in~(\ref{eq:u0v0}); consequently, $\|\nabla v_h^n\|_{L^\infty(\Omega)}$ does at least for the first time steps. Therefore,  computing $u^{n+1}_h$ using \eqref{eq:u_h} turns out to be more demanding. Figure \ref{fig:min_uh} (top) plots the values $\{\min_{\x\in\Omega}u_h^n(\x)\,, n=0,\dots,50\}$ for $\cteU\in\{70,80,90,100\}$. Positivity is recovered once $\|\nabla v_h^n\|_{L^\infty(\Omega)}$ becomes small enough.

This loss of positivity for large values of \rf{$\cteU$} is not in contradiction to \eqref{Global-Lower-Bound-uh} in Theorem~\ref{Th:main}, since \eqref{restriction-h-k} and \eqref{restriction-h} are not fulfilled for those cases\footnote{Since we do not know the exact value  of the constant $C$ in conditions \eqref{restriction-h-k} and \eqref{restriction-h}, we took $C$ to be $1$. On the other hand, for values of $(\cteU, \cteV)$ being small enough, both conditions hold since $F(u_0, v_0)$ decreases as $(\cteU, \cteV)$ do. So, positivity is maintained.}. Moreover it is remarkable that, for $\cteU\simeq 70$, $u_h^{n+1}$ keeps positivity, even when \eqref{restriction-h-k} and \eqref{restriction-h} are quite far from being verified. In fact, $\mathcal{F}(u_h^0, v_h^0)$ takes huge values, which exceed the capacity of floating point standards. These huge values stem from $\|\nabla v_h^0\| \simeq 7,687.66$, which gives ${\mathcal B}_2(u_h^0,v_h^0)\simeq 23,411.5$, used as an exponent for computing ${\mathcal F}(u_h^0,v_h^0)$. In this sense, our numerical experiments suggest that there might be room for improvement in conditions \eqref{restriction-h-k} and \eqref{restriction-h} of Theorem~\ref{Th:main}.

\begin{figure}
  \centering
  \begin{tabular}{c@{}c}
    \includegraphics[width=0.97\linewidth]{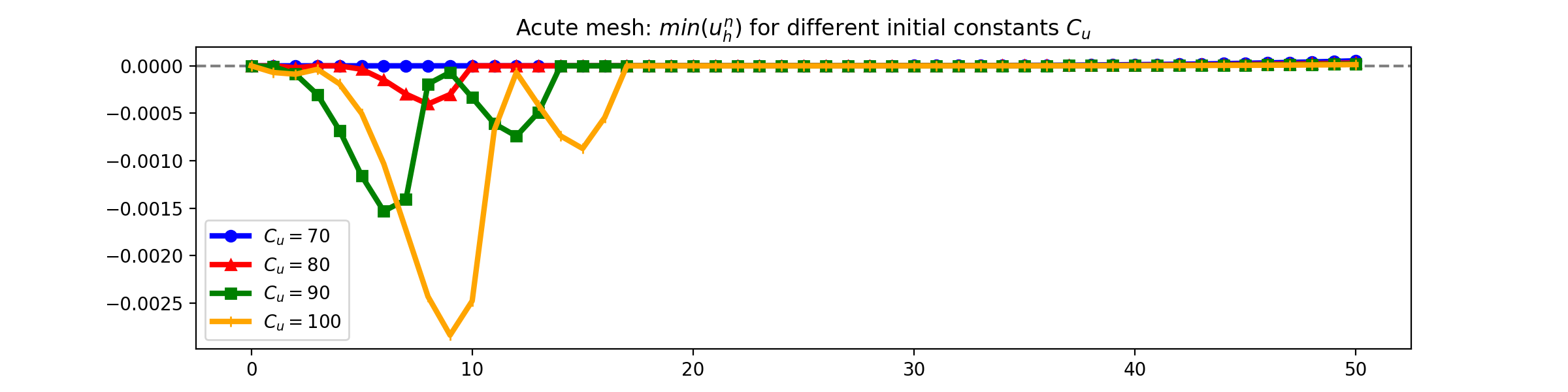}
    \\
  \includegraphics[width=0.97\linewidth]{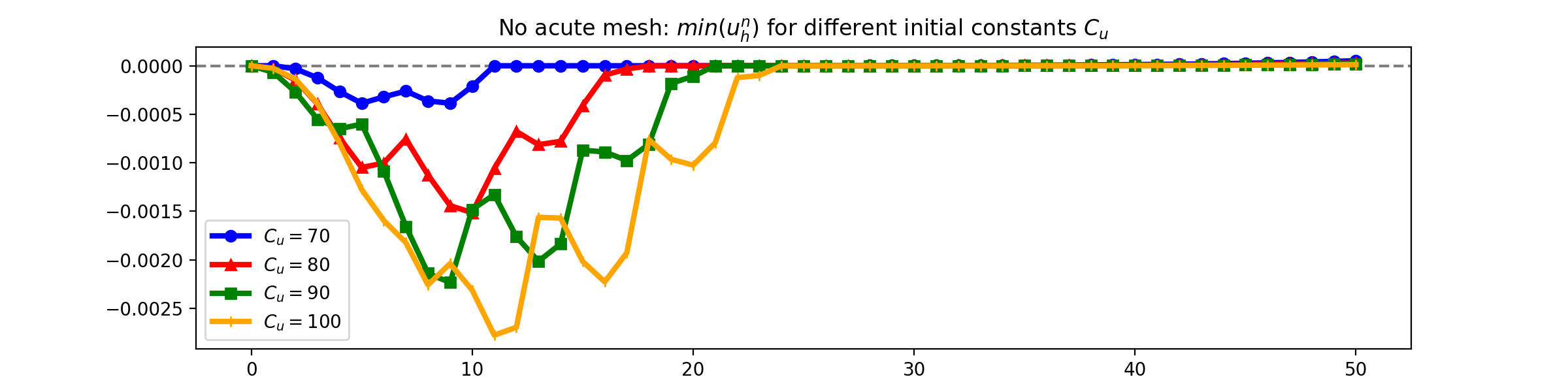}
  \end{tabular}
  \caption{Plot of $\min_{\mathcal{T}_h}(u_h^n)$,
    $n=0,\dots 50$, where ${\mathcal{T}_h}$ is an acute mesh (top) or
    non-acute mesh (bottom). Initial value constants:
    $\cteU=70,80,90,100$}.
  \label{fig:min_uh}
\end{figure}

In order to compare the performance of scheme \eqref{eq:u_h}--\eqref{eq:v_h} using a \emph{non-acute} mesh, we consider, as before, a mesh composed of $50\times 50$ macroelements as depicted in Figure \ref{fig:obtuse_macrolement}.  This way the theoretical results shown in this paper may not be applied.
\begin{figure}
  \centering
  \rf{\includegraphics[width=0.4\linewidth, height=0.37\linewidth]{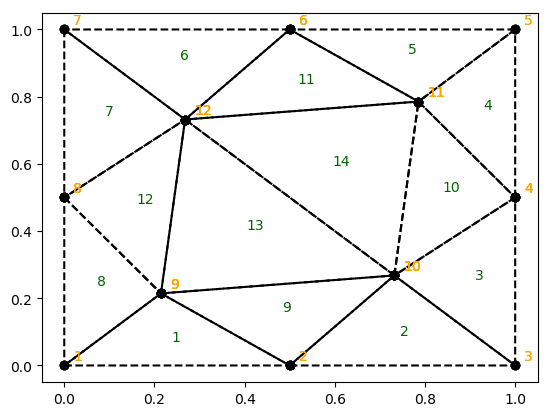}}
    \caption{Reference macrolement containing some obtuse triangles (triangles 1, 9, 5 and 11).}
  \label{fig:obtuse_macrolement}
\end{figure}

Diffusion and chemotaxis movements are obtained as observed in Figure \ref{fig:test1_cu70} for $\cteU=70$, but an earlier lost of positivity as well. In particular, it is lost from the first time step
($\min_{{\mathcal T}_h} (u^n_h) \simeq -1.39289\cdot 10^{-5}$ at $t_n=10^{-4}$). Positivity is not completely recovered until $t_n=0.0018$, thereafter positive values persist with time. Figure \ref{fig:min_uh} (bottom) displays the evolution of the values $\{\min_{\x\in\Omega}u_h^n(\x)\,, n=0,\dots,50\}$ for $\cteU\in\{70,80,90,100\}$.

\subsection{Blowup setting}
The second suite of tests is focused on a blowup context.  We consider
\begin{equation}
  u_0 = \cteU e^{-0.1\cteU(x^2+y^2)} \quad\mbox{ and }\quad  v_0 = \cteV e^{-0.1\cteV(x^2+y^2)},
  \label{eq:u0v0.test2}
\end{equation}
with $\cteU=1000$ and $\cteV=500$. Thus constructed, initial data are large enough to expect a finite time blowup for both $u$ and $v$ components of the continuous solution to problem \eqref{KS} with \eqref{IC}--\eqref{BC}. For details, see e.g.~\cite{Chertock_Kurganov_2008}, where the blowup time $t^*$ is conjectured to be located in the time interval $(4.4\cdot 10^{-5}, 10^{-4})$. 

When used an \textit{acute} mesh of macroelements as in Figure \ref{fig:acute_macrolement} for approximating such a demanding blowup test,  scheme \eqref{eq:u_h}--\eqref{eq:v_h} cannot aspire to achieve positivity over the whole blowup interval. The reason is that conditions \eqref{restriction-h-k} and \eqref{restriction-h} in Theorem~\ref{Th:main} are not fulfilled, because $\|u_h^0\|$ and $\|\nabla v_h^0\|$ are too large ($\|u_h^0\|_h^2 = 15,708$ and $\|\nabla v_h^0\|_{L^2(\Omega)}^2 = 785,230$). However, as in the previous experiments, one does not need \eqref{restriction-h-k} and \eqref{restriction-h} to hold so as to keep positivity.  For instance, a value $\Nsquare = 600$ suffices to obtain positivity for the overall blowup interval. Moreover, a value $\Nsquare=100$ ($h\simeq 0.005$) maintains positivity well into $(4.4\cdot 10^{-5}, 10^{-4})$. To be more precise, if ${\mathcal T}_h$ is defined by $100\times 100$ macroelements and $k=10^{-6}$ is chosen, then $u^{n}_h>0$ and $v^{n}_h>0$ for $t_n\in [0, 8.7\cdot 10^{-5})$. For $t_n=8.7\cdot 10^{-5}$ ($n=88$), one gets $\min_{{\mathcal T}_h}(u^n_h)=-90.7418$.
\begin{figure}
  \centering
  \begin{tabular}{c@{\rule{2em}{0pt}}c}
    \includegraphics[width=0.4\linewidth]{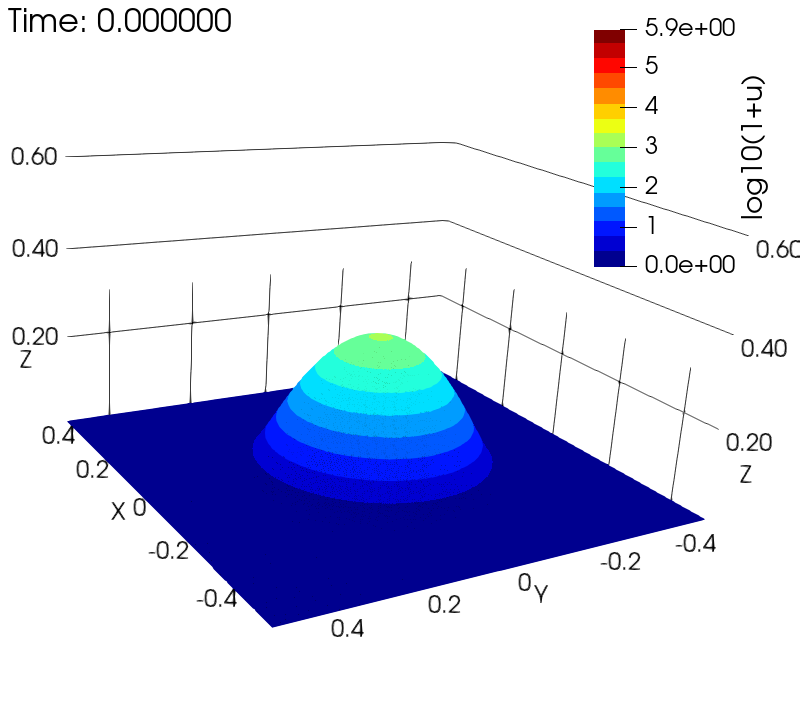}&
    \includegraphics[width=0.4\linewidth]{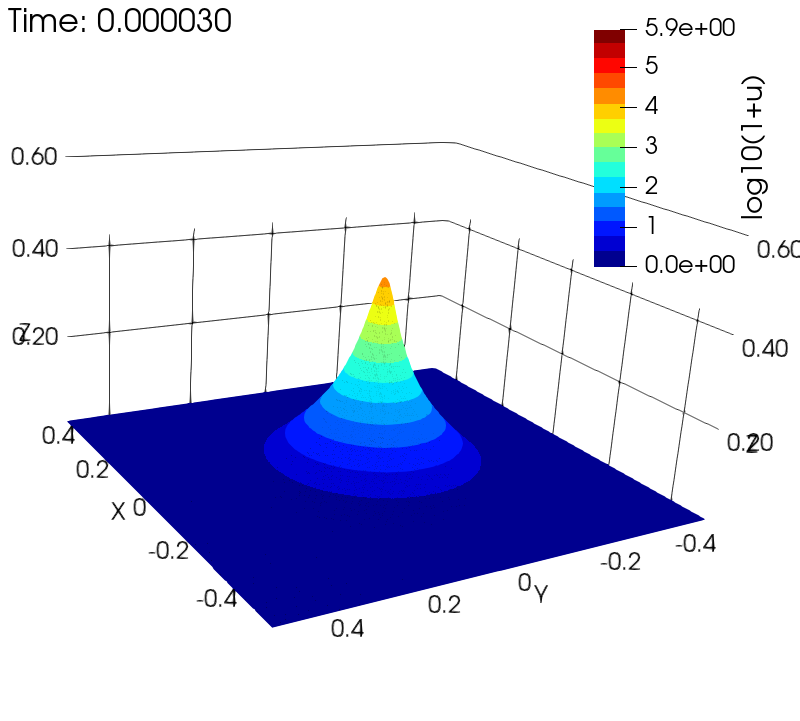}
    \\
  \includegraphics[width=0.4\linewidth]{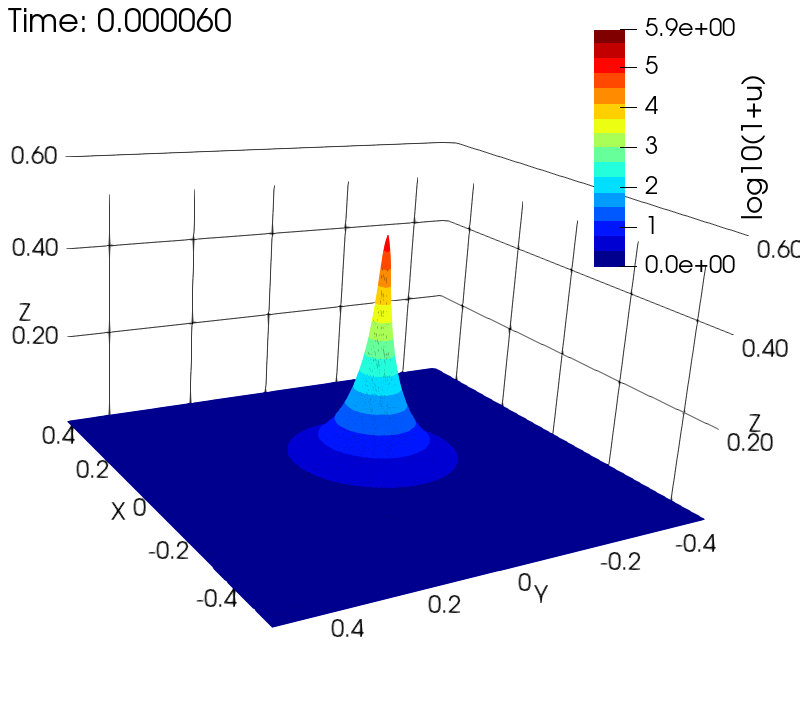}&
  \includegraphics[width=0.4\linewidth]{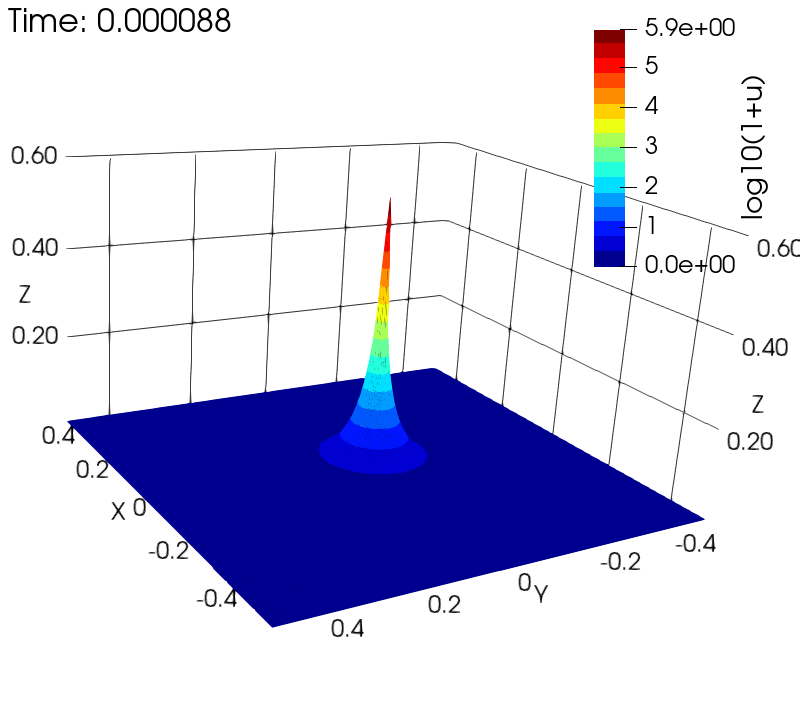}
  \end{tabular}
  \caption{Solution $u^n_h$ associated with an acute mesh of
   $100\times 100$ macroelements at time steps $n=0$, $30$, $60$, and $88$ (when positivity is broken).}
  \label{fig:blowup}
\end{figure}
The evolution of $u^n_h$ is shown (on a logarithmic scale) in Figure \ref{fig:blowup} for time steps $n=0$, $30$, $60$, and $88$. A blowup phenomenon in the center of the domain can be observed: the maximum value of $u^n_h$ grows over time, reaching $\max(u^n_h)=8.85515\times 10^5$, while its support shrinks.

When considering a \emph{non-acute} mesh of $100\times100$ macroelements (see Figure \ref{fig:obtuse_macrolement}), we encounter that  positivity is only maintained until $t_n=6.7\cdot 10^{-5}$.  At time $t_n=6.8\cdot 10^{-5}$ ($n=68$), $\min_{{\mathcal T}_h}(u^n_h)=-64.2157$ and $\max_{{\mathcal T}_h}(u^n_h)=2.16982\times 10^5$.
\begin{figure}
  \centering
  \begin{tabular}{@{}c@{\rule{.03\linewidth}{0pt}}c@{\rule{.03\linewidth}{0pt}}c@{}}
    \includegraphics[width=0.31\linewidth]{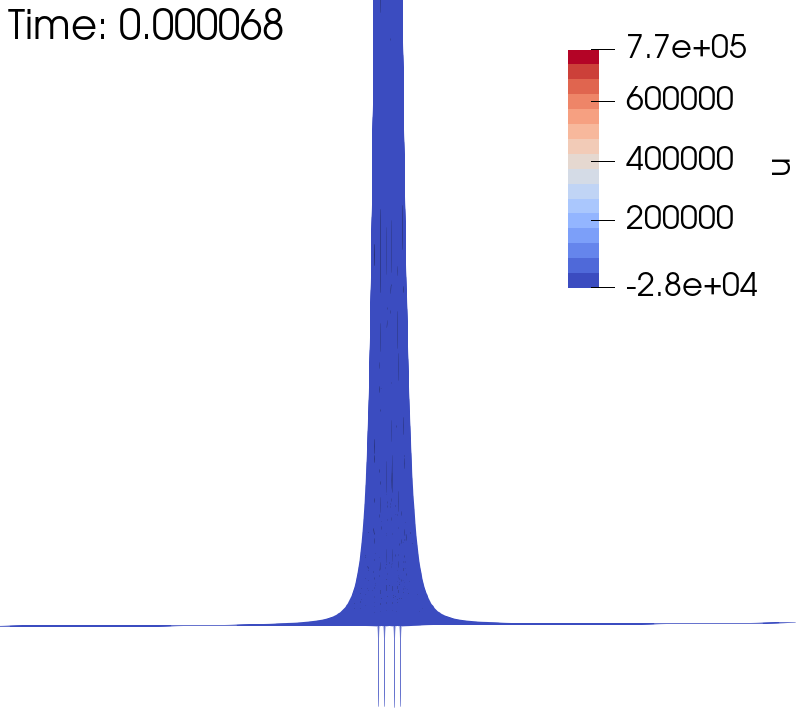}&
    \includegraphics[width=0.31\linewidth]{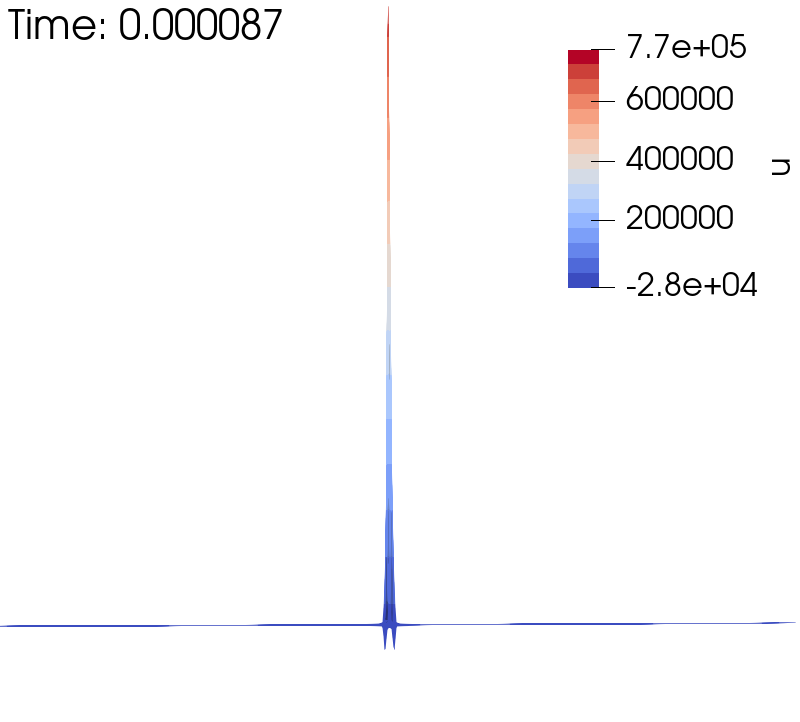}&
    \includegraphics[width=0.31\linewidth]{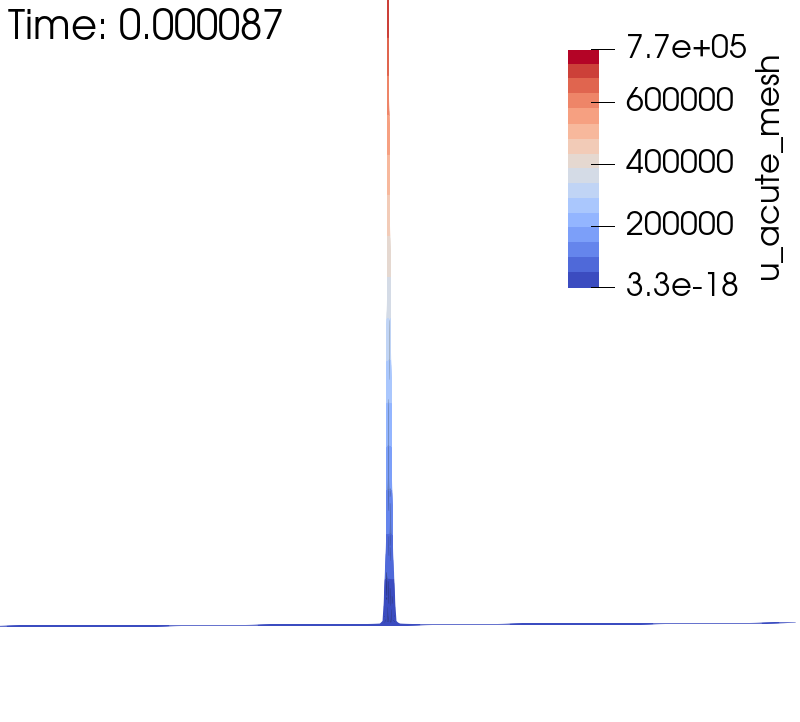}
  \end{tabular}
  \caption{ Left: zoomed detail of $u^n_h$ associated with a non-acute mesh at $t_n=6.8\cdot 10^{-5}$; positivity is broken.  Center:  $u^n_h$ associated with a  non-acute mesh at $t_n=8.7\cdot 10^{-5}$; deep negative values appear. Right: $u^n_h$ associated with an acute mesh at $t_n=8.7\cdot 10^{-5}$, positivity is maintained.}

  \label{fig:positivity}
\end{figure}
Figure~\ref{fig:positivity} shows the numerical solution $u^n_h$ at $t_n=6.8\cdot 10^{-5}$ (left), when positivity is broken for the first time, and at $t_n=8.7\cdot 10^{-5}$ (center), when negative values of order $10^{-4}$ are reached. Otherwise, the numerical solution $u^n_h$ associated with an acute mesh keeps positivity at time $t_n=8.7\cdot 10^{-5}$ (right).

\end{document}